\documentclass{birkjour}

\usepackage[utf8]{inputenc}
\usepackage[T1]{fontenc}
\usepackage[english]{babel}
\usepackage{mathtools,amsmath,amssymb,amsfonts,amsthm,mathrsfs}
\usepackage{bbm}
\mathtoolsset{showonlyrefs}
\usepackage{enumitem}

\usepackage{geometry}
\usepackage{graphicx,color}
\usepackage{tikz,pgfplots,pgf}
\usepackage[caption=false]{epsfig}
\usepackage{subcaption}
\usepackage[font=small]{caption}

\newtheorem{theorem}{Theorem}[section]

\newtheorem{lemma}{Lemma}[section]

\newtheorem{corollary}{Corollary}[section]

\theoremstyle{definition}
\newtheorem{remark}{Remark}[section]

\numberwithin{equation}{section}

\begin{document}

\title[High multiplicity and chaos for an indefinite problem]{High multiplicity and chaos for an indefinite \\ problem arising from genetic models}

\author[A.~Boscaggin]{Alberto Boscaggin}

\address{
Department of Mathematics ``Giuseppe Peano'', University of Torino\\
Via Carlo Alberto 10, 10123 Torino, Italy}

\email{alberto.boscaggin@unito.it}

\author[G.~Feltrin]{Guglielmo Feltrin}

\address{
Department of Mathematics, Computer Science and Physics, University of Udine\\
Via delle Scienze 206, 33100 Udine, Italy}

\email{guglielmo.feltrin@uniud.it}

\author[E.~Sovrano]{Elisa Sovrano}

\address{
Istituto Nazionale di Alta Matematica ``Francesco Severi'' c/o
Department of Mathematics and Geosciences, University of Trieste\\
Via Valerio 12/1, 34127 Trieste, Italy}

\email{esovrano@units.it}

\thanks{Work written under the auspices of the Grup\-po Na\-zio\-na\-le per l'Anali\-si Ma\-te\-ma\-ti\-ca, la Pro\-ba\-bi\-li\-t\`{a} e le lo\-ro Appli\-ca\-zio\-ni (GNAMPA) of the Isti\-tu\-to Na\-zio\-na\-le di Al\-ta Ma\-te\-ma\-ti\-ca (INdAM).
The first two authors are supported by the project ERC Advanced Grant 2013 n.~339958 ``Complex Patterns for Strongly Interacting Dynamical Systems - COMPAT''.
The first author is supported by INdAM-GNAMPA project ``Il modello di Born--Infeld per l'elettromagnetismo nonlineare: esistenza, regolarit\`{a} e molteplicit\`{a} di soluzioni''.
The third author is supported by INdAM project ``Problems in Population Dynamics: from Linear to Nonlinear Diffusion''.
\\
\textbf{Preprint -- May 2019}} 



\date{}

\dedicatory{}

\begin{abstract}
We deal with the periodic boundary value problem associated with the parameter-dependent second-order nonlinear differential equation
\begin{equation*}
u'' + cu' + \bigr{(} \lambda a^{+}(x) - \mu a^{-}(x) \bigr{)} g(u) = 0,
\end{equation*}
where $\lambda,\mu>0$ are parameters, $c\in\mathbb{R}$, $a(x)$ is a locally integrable $P$-periodic sign-changing weight function, and $g\colon\mathopen{[}0,1\mathclose{]}\to\mathbb{R}$ is a continuous function such that $g(0)=g(1)=0$, $g(u)>0$ for all $u\in\mathopen{]}0,1\mathclose{[}$, with superlinear growth at zero. A typical example for $g(u)$, that is of interest in population genetics, is the logistic-type nonlinearity $g(u)=u^{2}(1-u)$.

Using a topological degree approach, we provide high multiplicity results by exploiting the nodal behaviour of $a(x)$. More precisely, when $m$ is the number of intervals of positivity of $a(x)$ in a $P$-periodicity interval, we prove the existence of $3^{m}-1$ non-constant positive $P$-periodic solutions, whenever the parameters $\lambda$ and $\mu$ are positive and large enough. 
Such a result extends to the case of subharmonic solutions. Moreover, by an approximation argument, we show the existence of a countable family of globally defined solutions with a complex behaviour, coded by (possibly non-periodic) bi-infinite sequences of $3$ symbols.
\end{abstract}

\maketitle

\clearpage

\section{Introduction and statement of the results}\label{section-1}

In this paper, we investigate existence and multiplicity of non-constant positive solutions for the parameter-dependent second-order ordinary differential equation
\begin{equation*}
u'' + c u'+\bigl{(}\lambda a^{+}(x)-\mu a^{-}(x)\bigr{)} g(u) = 0,
\leqno{(\mathscr{E}_{\lambda,\mu})}
\end{equation*}
where $\lambda$ and $\mu$ are positive real parameters, $c \in \mathbb{R}$, $a^{+}(x)$ and $a^{-}(x)$ are the positive and the negative part, respectively, of a $P$-periodic and locally integrable sign-changing function $a \colon \mathbb{R} \to \mathbb{R}$, and $g\colon\mathopen{[}0,1\mathclose{]}\to\mathbb{R}$ is a continuous map satisfying the sign condition 
\begin{equation*}
g(0) = g(1) = 0, \qquad g(u) > 0 \quad \text{for all $u\in\mathopen{]}0,1\mathclose{[}$},
\leqno{(g_{*})}
\end{equation*}
and the superlinear growth condition at zero
\begin{equation*}
\lim_{u\to 0^{+}} \dfrac{g(u)}{u} = 0.
\leqno{(g_{0})}
\end{equation*}
Following a terminology popularized in \cite{HeKa-80}, we refer to $(\mathscr{E}_{\lambda,\mu})$ as an \textit{indefinite} equation, meaning that the weight function $a(x)$ changes sign. In the last decades this kind of equations has been widely investigated, both in the ODE and in the PDE setting, starting from the classical contributions \cite{AlTa-96,AmLG-98,BaPoTe-88,BeCDNi-94,Bu-76}
and till to very recent ones \cite{BoFe-PP,DoZa-PP,HaZa-17,LGOmRi-17,SoZa-17,Te-18,Ur-17}; we refer the reader to \cite{Fe-18} for a quite exhaustive bibliography on the subject.

The mathematical questions we address here are motivated by the study of the spatial effects on the variation in the genetic material along environmental gradients. In population genetics, when individuals of a continuously distributed population mate at random in their habitat, and no genetic drift nor new mutations appear, the evolution of the frequencies of two alleles, $A_{1}$ and $A_{2}$, at a single locus under the action of migration and selection can be described through the reaction-diffusion boundary value problem
\begin{equation}\label{eq-gen}
\begin{cases}
\, \partial_{t} u =\sum_{i,j} V_{i,j}(x)\partial_{x_{i}x_{j}} u +b(x)\cdot \nabla u+h(x,u)	&\text{in $\Omega\times \mathopen{]}0,+\infty\mathclose{[}$,}\\
\, 0\leq u\leq 1 	&\text{in $\Omega\times \mathopen{]}0,+\infty\mathclose{[}$,}\\
\, \nu(x) 	\cdot V(x) \nabla u=0		&\text{on $\partial\Omega\times \mathopen{]}0,+\infty\mathclose{[}$,}
\end{cases}
\end{equation}
where $u(x, t)$ and $1-u(x, t)$ denote the allele frequency of $A_{1}$ and $A_{2}$, respectively (cf.~\cite{LNN-13,Na-89}).
The set $\Omega\subset\mathbb{R}^{N}$ ($N\geq 1$) represents the habitat that is assumed to be a bounded domain with smooth boundary $\partial\Omega$
and outward unit normal vector $\nu(x)$. The matrix-valued function $V(x)$ and the vector-valued function $b(x)$ are given and characterize the migration. Finally, $h(x,u)$ is a nonlinear term which describes the effects of the selection and satisfies $h(x,0)=0=h(x,1)$ for all $x\in\Omega$, so that $u\equiv0$ and $u\equiv1$ are constant solutions of problem \eqref{eq-gen} that means that allele $A_{1}$ is absent or is fixed in the population, respectively. 

In this context, available theory also assumes that migration is homogeneous and isotropic, namely, $V(x)$ is constant and $b\equiv0$, and that the selection is of the form $h(x,u)=a(x)g(u)$, where $a(x)$ is the spatial factor and $g(u)$ is a function of gene frequency satisfying $(g_{*})$. The sign-indefinite weight term $a(x)$ reflects at least one change in the direction of selection and leads to several environmental regions in the habitat $\Omega$ which are favorable ($a(x)>0$), neutral ($a(x)=0$), or unfavorable ($a(x)<0$) for one allele. In this connection, investigations on non-constant positive stationary solutions (i.e.,~clines) lead to the study of the Neumann problem
\begin{equation}\label{eq-gen2}
\begin{cases}
\, d\Delta u+a(x)g(u)=0 &\text{in $\Omega$,}\\
\, 0\leq u\leq 1 &\text{in $\partial\Omega$,}\\
\, \partial_{\nu} u =0 &\text{on $\partial\Omega$,}
\end{cases}
\end{equation}
where $\Delta$ denotes the Laplace operator and $d>0$ is the diffusion rate. Neumann boundary conditions model an impenetrable barrier for the population so that no-flux of genes across the boundary occurs. The number and the stability of non-constant positive solutions of \eqref{eq-gen2} are governed by the features of both the components $a(x)$ and $g(u)$. 

The existence of a unique non-constant and globally asymptotically stable solution of \eqref{eq-gen2} is proved in~\cite{BrHe-90,He-81,LoNa-02} for sufficiently small $d$ provided that $\int_\Omega a(x)\,\mathrm{d}x < 0$ and $g(u)$ is a smooth function such that $g''(u)<0$ for every 
$u\in \mathopen{]}0,1\mathclose{[}$. The archetypical example is the case when no allele is dominant or the population is haploid, namely $g(u)=u(1-u)$ (e.g.,~\cite{Ha-48,Na-75}). On the other hand, if $g(u)$ is not concave, multiplicity results for \eqref{eq-gen2} are shown in \cite{LNS-10,So-17jomb}. In particular, if $g'(0)=0$ and we assume also that $\lim_{u\to 0^{+}}g(u)/u^{k}>0$ for some $k > 1$, then for $d$ sufficiently small there exist at least two non-constant solutions: one stable and the other unstable (cf.~\cite[Theorem~2.9]{LNS-10}). The main example in this framework concerns completely dominance of allele $A_{2}$ over allele $A_{1}$, namely $g(u)=u^{2}(1-u)$ (e.g.,~\cite{LNN-13,LNS-10}). 

In this paper, we deal with migration-selection models in a unidimensional habitat. We also assume that $V(x)$ and $b(x)$ are constant functions, with $b(x)=c$ for some $c\in\mathbb{R}$. Moreover, we describe the strength of selection in the environmental regions which are beneficial or harmful for the alleles by introducing two positive independent parameters, $\lambda$ and $\mu$, on which we discharge the migration rate. Precisely, the weight term we consider is defined as
\begin{equation}\label{eq-weight}
a_{\lambda,\mu}(x) :=\lambda a^{+}(x)-\mu a^{-}(x).
\end{equation}
Hence, the selection is $h(x,u)=a_{\lambda,\mu}(x)g(u)$ where $g(u)$ satisfies $(g_{*})$ and, in order to include recessive phenomena as a case study, we assume also condition $(g_{0})$. In such a way, we are lead to equation $(\mathscr{E}_{\lambda,\mu})$. We notice that, for $\lambda = \mu = 1/d$ and $c = 0$, this gives the one-dimensional version of the elliptic PDE in \eqref{eq-gen2}.

We are interested in periodically changes in genotype within a population as a function of spatial location. Thus we 
assume that $a(x)$ is $P$-periodic (for some $P > 0$) and we seek non-constant positive solutions of equation $(\mathscr{E}_{\lambda,\mu})$ (in the Carath\'{e}odory sense, see \cite[Section~I.5]{Ha-80}) satisfying \textit{periodic boundary conditions}
\begin{equation}
u(0)=u(P),\quad u'(0)=u'(P).
\end{equation} 
These models are appropriate in the case of populations living in circular habitats (e.g.,~around a lake or along the shore of an island), as well as for ring species, for instance, around the arctic.

To state our main results, we introduce the following condition on the weight function $a(x)$ that we assume henceforth:
\begin{itemize}
\item[$(a_{*})$]
\textit{there exist $m \geq 1$ non-empty closed intervals $I^{+}_{1},\ldots,I^{+}_{m}$
separated by $m$ non-empty closed intervals $I^{-}_{1},\ldots,I^{-}_{m}$ such that
\begin{equation*}
\bigcup_{i=1}^{m} I^{+}_{i} \, \cup \, \bigcup_{i=1}^{m} I^{-}_{i} = \mathopen{[}0,P\mathclose{]},
\end{equation*}
and
\begin{equation*}
a(x)\succ 0 \; \text{ on $I^{+}_{i}$}, \qquad a(x)\prec 0 \; \text{ on $I^{-}_{i}$}.
\end{equation*}}
\end{itemize}
In the above condition, the symbol $\succ$ (respectively, $\prec$) means that $a(x) \geq 0$ (respectively, $a(x) \leq 0$), with $a(x) \not\equiv 0$.
We also define
\begin{equation}\label{def-musharp}
\mu^{\#}(\lambda) := \lambda \, \dfrac{\int_{0}^{P} a^{+}(x)\,\mathrm{d}x}{\int_{0}^{P} a^{-}(x)\,\mathrm{d}x}
\end{equation}
and notice that $\int_{0}^{P} a_{\lambda,\mu}(x)\,\mathrm{d}x < 0$ if and only if $\mu > \mu^{\#}(\lambda)$. 

With this notation, our first result reads as follows.

\begin{theorem}\label{th-exis}
Let $c\in\mathbb{R}$ and let $a \colon \mathbb{R} \to \mathbb{R}$ be a $P$-periodic locally integrable function satisfying $(a_{*})$. Let $g \colon \mathopen{[}0,1\mathclose{]} \to \mathbb{R}$ be a continuously differentiable function satisfying $(g_{*})$ and $(g_{0})$.
Then, there exists $\lambda^{*} > 0$ such that for every $\lambda > \lambda^{*}$ and for every $\mu>\mu^{\#}(\lambda)$ equation $(\mathscr{E}_{\lambda,\mu})$ has at least two non-constant positive $P$-periodic solutions.

More precisely, fixed an arbitrary constant $\rho\in\mathopen{]}0,1\mathclose{[}$ there exists $\lambda^{*} = \lambda^{*}(\rho) > 0$ such that for every $\lambda > \lambda^{*}$ and for every $\mu>\mu^{\#}(\lambda)$ there exist two positive $P$-periodic solutions $u_{s}(x)$ and $u_{\ell}(x)$ to $(\mathscr{E}_{\lambda,\mu})$ such that
\begin{equation*}
0<\|u_{s}\|_{\infty}<\rho<\|u_{\ell}\|_{\infty}<1.
\end{equation*}
\end{theorem}

Let us notice that, when $\int_{0}^{P} a(x)\,\mathrm{d}x < 0$, an application of Theorem~\ref{th-exis} with $\mu = \lambda$ provides two non-constant positive $P$-periodic solutions of the one-parameter equation
\begin{equation}\label{eq-lambda}
u'' + cu' + \lambda a(x)g(u) = 0,
\end{equation}
for $\lambda > 0$ sufficiently large (see Corollary \ref{cor-exis}). When $c = 0$, this result can thus be interpreted as a periodic version of the two-solution theorem given in~\cite[Theorem~2.9]{LNS-10} for the Neumann boundary value problem (indeed, $\lambda = 1/d$ large implies $d$ small). It is remarkable, however, that the same result holds even in the non-Hamiltonian case $c \neq 0$.  

The second, and main, part of our investigation is focused on the appearance of \textit{high multiplicity} phenomena for solutions of $(\mathscr{E}_{\lambda,\mu})$. In this regard, the fact that the weight function $a_{\lambda,\mu}(x)$ defined in \eqref{eq-weight} depends on two parameters $\lambda$ and $\mu$ plays a crucial role: indeed, high multiplicity of periodic solutions will be proved to arise when $\lambda > \lambda^*$ is fixed (where $\lambda^*$ is the constant already given by Theorem~\ref{th-exis}) and $\mu$ is sufficiently large (typically, much larger than the constant $\mu^{\#}(\lambda)$ defined in \eqref{def-musharp}).

To state our result precisely, we introduce the condition 
\begin{equation*}
\limsup_{u\to 1^{-}} \dfrac{g(u)}{1-u} <+\infty
\leqno{(g_{1})}
\end{equation*}
and notice that it is satisfied whenever $g(u)$ is continuously differentiable in a left neighborhood of $u = 1$.
To complement Theorem~\ref{th-exis} we have the following result. We remark that an analogous result is also valid if Dirichlet or Neumann boundary conditions are considered (see Section~\ref{section-6.2}). 

\begin{theorem}\label{th-mult}
Let $c\in\mathbb{R}$ and let $a \colon \mathbb{R} \to \mathbb{R}$ be a $P$-periodic locally integrable function satisfying $(a_{*})$. Let $g \colon \mathopen{[}0,1\mathclose{]} \to \mathbb{R}$ be a continuous function satisfying $(g_{*})$, $(g_{0})$ and $(g_{1})$. 
Then, there exists $\lambda^{*} > 0$ such that for every $\lambda > \lambda^{*}$ there exists $\mu^{*}(\lambda) > 0$ such that
for every $\mu > \mu^{*}(\lambda)$ equation $(\mathscr{E}_{\lambda,\mu})$ has at least $3^{m}-1$ non-constant positive $P$-periodic solutions.

More precisely, fixed an arbitrary constant $\rho\in\mathopen{]}0,1\mathclose{[}$ there exists $\lambda^{*} = \lambda^{*}(\rho) > 0$ such that for every $\lambda > \lambda^{*}$ there exist two constants $r,R$ with $0 < r < \rho < R < 1$ and $\mu^{*}(\lambda) =\mu^{*}(\lambda,r,R)>0$ such that for every $\mu > \mu^{*}(\lambda)$ and for every finite string $\mathcal{S} = (\mathcal{S}_{1},\ldots,\mathcal{S}_{m}) \in \{0,1,2\}^{m}$, with $\mathcal{S} \neq (0,\ldots,0)$,
there exists at least one positive $P$-periodic solution $u_{\mathcal{S}}(x)$ of $(\mathscr{E}_{\lambda,\mu})$ such that
\begin{itemize}
\item $\max_{x \in I^{+}_{i}} u_{\mathcal{S}}(x) < r$, if $\mathcal{S}_{i} = 0$;
\item $r < \max_{x \in I^{+}_{i}} u_{\mathcal{S}}(x) < \rho$, if $\mathcal{S}_{i} = 1$;
\item $\rho < \max_{x \in I^{+}_{i}} u_{\mathcal{S}}(x) < R$, if $\mathcal{S}_{i} = 2$;
\end{itemize}
for every $i=1,\ldots,m$.
\end{theorem}

Let us notice that the number of solutions provided by Theorem~\ref{th-mult} is strongly related with the nodal behavior of the weight function $a_{\lambda,\mu}(x)$: the larger the number of nodal domains of the weight function, $m$, the greater the number of solutions obtained, $3^m-1$. Observe also that the number $3^m-1$ comes from the possibility of ``coding'' the solutions via their behavior in each interval of positivity $I^{+}_{i}$: ``very small'' ($\mathcal{S}_{i} = 0$), ``small'' ($\mathcal{S}_{i} = 1$) or ``large'' ($\mathcal{S}_{i} = 2$).  We mention that the same type of multiplicity pattern also emerges in a different context, namely for equation $(\mathscr{E}_{\lambda,\mu})$ with $c = 0$ and a nonlinear term $g\colon \mathopen{[}0,+\infty\mathclose{[} \to \mathopen{[}0,+\infty\mathclose{[}$ satisfying $(g_{0})$ and having sublinear growth at infinity, that is, $g(u)/u \to 0$ for $u \to +\infty$ (see \cite{BoFeZa-18tams}).

The possibility of providing, in the context of indefinite boundary value problems, high multiplicity results by playing with the nodal behavior of the weight function was first suggested in \cite{GRLG-00}; therein, an interesting analogy was proposed with the papers \cite{Da-88,Da-90}, giving, in the PDE setting, multiplicity of solutions depending on the shape of the domain. 
Later on, along this line of research, several contributions followed \cite{BoGoHa-05,Bo-11,BoDaPa-17,BoFeZa-18tams,FeSo-18non,FeSo-18na,FeZa-15jde,FeZa-17jde,GaHaZa-03,GaHaZa-04}. 
In particular, dealing with equation $(\mathscr{E}_{\lambda,\mu})$, with $c= 0$ and $g(u)$ a Lipschitz continuous function satisfying $(g_{*})$ and $(g_{0})$, the existence of $8 = 3^{2} - 1$ positive solutions for both the Dirichlet and the Neumann boundary value problem was previously proved in \cite{FeSo-18non}, for a weight function $a(x)$ with $m=2$ intervals of positivity. Therefore, Theorem~\ref{th-mult} extends the result therein to the general case $m \geq 2$ and to a wider class of boundary conditions, including periodic ones, possibly in the non-Hamiltonian case $c \neq 0$. It is worth noticing that this was explicitly raised as an open problem in \cite[Conjecture 2]{FeSo-18non}; let us stress however that
the shooting arguments employed in \cite{FeSo-18non} by no means can be used to investigate the periodic problem, and in the present paper we rely on a completely different approach. 

Our last result concerns the dynamics of equation $(\mathscr{E}_{\lambda,\mu})$ on the whole real line.
Precisely, having defined the intervals
\begin{equation*}
I^{+}_{i,\ell} := I^{+}_{i} + \ell P, \quad i=1,\ldots,m, \; \ell\in\mathbb{Z},
\end{equation*}
we provide globally defined positive solutions of $(\mathscr{E}_{\lambda,\mu})$, whose behavior in each of the above intervals can be coded, as in Theorem~\ref{th-mult}, by a bi-infinite (possibly non-periodic) sequence $\mathcal{S} \in\{0,1,2\}^{\mathbb{Z}}$. This is
a picture of symbolic dynamics, and equation $(\mathscr{E}_{\lambda,\mu})$ is said to exhibit \textit{chaos}.
The precise statement is the following.

\begin{theorem}\label{th-chaos}
Let $c\in\mathbb{R}$ and let $a \colon \mathbb{R} \to \mathbb{R}$ be a locally integrable periodic function of minimal period
$P > 0$ satisfying $(a_{*})$. Let $g \colon \mathopen{[}0,1\mathclose{]} \to \mathbb{R}$ be a continuous function satisfying $(g_{*})$, $(g_{0})$ and $(g_{1})$.
Then, fixed an arbitrary constant $\rho\in\mathopen{]}0,1\mathclose{[}$ there exists $\lambda^{*} = \lambda^{*}(\rho) > 0$ such that
for every $\lambda > \lambda^{*}$ there exist two constants $r$ and $R$ with $0 < r < \rho < R < 1$, and $\mu^{*}(\lambda) =
\mu^{*}(\lambda,r,R)>0$ such that for every $\mu > \mu^{*}(\lambda)$ the following holds:
given any two-sided sequence
$\mathcal{S} = (\mathcal{S}_{j})_{j\in\mathbb{Z}}\in\{0,1,2\}^{\mathbb{Z}}$ which is not identically zero,
there exists at least one positive solution $u_{\mathcal{S}}(x)$ of $(\mathscr{E}_{\lambda,\mu})$ such that
\begin{itemize}
\item $\max_{x \in I^{+}_{i,\ell}} u_{\mathcal{S}}(x) < r$, if $\mathcal{S}_{i + \ell m} = 0$;
\item $r < \max_{x \in I^{+}_{i,\ell}} u_{\mathcal{S}}(x) < \rho$, if $\mathcal{S}_{i + \ell m} = 1$;
\item $\rho < \max_{x \in I^{+}_{i,\ell}} u_{\mathcal{S}}(x) < R$, if $\mathcal{S}_{i + \ell m} = 2$;\end{itemize}
for every $i=1,\ldots,m$ and $\ell\in\mathbb{Z}$.
In particular, if the sequence $\mathcal{S}$ is $km$-periodic for some integer $k\geq1$, there exists at least a positive $kP$-periodic solution $u_{\mathcal{S}}(x)$ of $(\mathscr{E}_{\lambda,\mu})$ satisfying the above properties.
\end{theorem}

For the proofs of Theorem~\ref{th-exis} and Theorem~\ref{th-mult}, we adopt a functional analytic approach based on topological degree theory in Banach spaces (cf.~\cite{FeZa-15jde} and the subsequent papers \cite{BoFeZa-16,BoFeZa-18tams,FeZa-17jde}). In particular, we follow the general strategies developed in \cite{BoFeZa-16,BoFeZa-18tams}, dealing with a nonlinear term $g \colon \mathopen{[}0,+\infty\mathclose{[} \to \mathopen{[}0,+\infty\mathclose{[}$ satisfying $(g_{0})$ and having sublinear growth at infinity. As already mentioned, these (super-sublinear) nonlinearities have similar features with respect to logistic-type nonlinearities considered in the present paper. However, while in the former case it is often possible to develop dual arguments for small/large solutions, here the presence of the constant solution $u \equiv 1$
leads to an ``asymmetric'' situation which requires completely new arguments. An important feature of this method of proof is that the estimates leading to the constant $\lambda^{*}$ and $\mu^{*}(\lambda)$ are fully explicit, depending only on the local behavior of the weight function $a(x)$ but not on the length of the periodicity interval. As a consequence, one can prove Theorem~\ref{th-chaos} via an approximation argument.

\smallskip

The paper is structured as follows. In Section~\ref{section-2}, we describe the abstract degree setting and we prove some technical estimates
on the solutions of $(\mathscr{E}_{\lambda,\mu})$ (and of some related equations). Based on this, in Section~\ref{section-3} and Section~\ref{section-4}, we give the proofs of Theorem~\ref{th-exis} and Theorem~\ref{th-mult}, respectively. The proof of Theorem~\ref{th-chaos} is then 
presented, together with some comments about the existence of subharmonic solutions, in Section~\ref{section-5}. The paper ends
with Section~\ref{section-6}, discussing some related results: subharmonic solutions via the Poincar\'e--Birkhoff theorem, Dirichlet/Neumann boundary value problems, stability issues, and an asymptotic analysis of the solutions for $\mu \to +\infty$.

\section{Abstract degree setting and technical lemmas}\label{section-2}

The aim of this section is to present the main tools used in the proofs of our theorems as well as some preliminary technical lemmas.

Before doing this, we introduce the following notation employed throughout the paper:
\begin{equation}\label{Ipm}
I^{+}_{i}=\mathopen{[}\sigma_{i},\tau_{i}\mathclose{]}, \qquad I^{-}_{i}=\mathopen{[}\tau_{i},\sigma_{i+1}\mathclose{]}, \qquad i=1,\ldots,m,
\end{equation}
where $\sigma_{i}$ and $\tau_{i}$ are suitable points such that
\begin{equation*}
0 = \sigma_{1} < \tau_{1} < \sigma_{2} < \tau_{2} < \ldots < \tau_{m-1} < \sigma_{m} < \tau_{m} < \sigma_{m+1} = P.
\end{equation*}
Notice that, due to the $P$-periodicity, we have assumed without loss of generality that $0 \in I^{+}_{1}$ (and, thus, $P \in I^{-}_{m}$).
We also stress that, in dealing with the above intervals, a cyclic convention will be adopted. For example, we will freely write expressions like $I^{-}_{i-1}\cup I^{+}_{i} \cup I^{-}_{i}$, where, if $i=1$, we agree that the interval $I^{-}_{0}$ means the $P$-shifted interval $I^{-}_{m}-P$. A similar remark applies for instance for $I^{+}_{i}\cup I^{-}_{i}\cup I^{+}_{i+1}$ when $i=m$ and, in such a case, $I^{+}_{m+1}=I^{+}_{1}+P$. This is not restrictive since the weight function $a(x)$ is $P$-periodic.

\subsection{Coincidence degree framework}\label{section-2.1}

In this section we recall Mawhin's coincidence degree theory (cf.~\cite{GaMa-77,Ma-79,Ma-93}) and we present two lemmas for the computation of the degree (cf.~\cite{BoFeZa-18tams}).

First of all, we remark that solving the $P$-periodic problem associated with $(\mathscr{E}_{\lambda,\mu})$ is equivalent to looking for solutions $u(x)$ of $(\mathscr{E}_{\lambda,\mu})$ defined on $\mathopen{[}0,P\mathclose{]}$ and such that $u(0) = u(P)$ and $u'(0) = u'(P)$.
Accordingly, let $X:=\mathcal{C}(\mathopen{[}0,P\mathclose{]})$ be the Banach space of continuous functions $u \colon \mathopen{[}0,P\mathclose{]} \to \mathbb{R}$,
endowed with the $\sup$-norm $\|u\|_{\infty} := \max_{x\in \mathopen{[}0,P\mathclose{]}} |u(x)|$,
and let $Z:=L^{1}(0,P)$ be the Banach space of integrable functions $v \colon \mathopen{[}0,P\mathclose{]} \to \mathbb{R}$,
endowed with the $L^{1}$-norm $\|v\|_{L^{1}(0,P)}:= \int_{0}^{P} |v(x)|\,\mathrm{d}x$.
We define the linear Fredholm map of index zero $L(u):= - u''-cu'$ on $\mathrm{dom}\,L := \bigl{\{}u\in W^{2,1}(0,P) \colon u(0) = u(P), \; u'(0) = u'(P) \bigr{\}} \subseteq X$. 
We also introduce the $L^{1}$-Carath\'{e}odory function
\begin{equation*}
f_{\lambda,\mu}(x,u) :=
\begin{cases}
\, -u, & \text{if $u \leq 0$,}\\
\, a_{\lambda,\mu}(x) g(u), & \text{if $u\in\mathopen{[}0,1\mathclose{]}$},\\
\, 0, & \text{if $u\geq 1$},\\
\end{cases}
\end{equation*}
and we denote by $N_{\lambda,\mu} \colon X \to Z$ the Nemytskii operator induced by the function $f_{\lambda,\mu}$, namely
\begin{equation*}
(N_{\lambda,\mu} u)(x):= f_{\lambda,\mu}(x,u(x)), \quad x\in \mathopen{[}0,P\mathclose{]}.
\end{equation*}

The coincidence degree theory ensures that the $P$-periodic problem associated with
\begin{equation}\label{eq-flm}
u'' + cu' +  f_{\lambda,\mu}(x,u) = 0
\end{equation} 
is equivalent to the coincidence equation
\begin{equation*}
Lu = N_{\lambda,\mu}u,\quad u\in \text{\rm dom}\,L,
\end{equation*}
or to the fixed point problem 
\begin{equation*}
u = \Phi_{\lambda,\mu}u:= \Pi u + QN_{\lambda,\mu}u + K_{\Pi}(\mathrm{Id}-Q)N_{\lambda,\mu}u, \quad u \in X,
\end{equation*}
where $\Pi \colon X \to \ker L \cong {\mathbb{R}}$, $Q \colon Z \to \text{\rm coker}\,L \cong Z/\mathrm{Im}\,L \cong \mathbb{R}$ are two projections, and $K_{\Pi} \colon \text{\rm Im}\,L \to \text{\rm dom}\,L \cap \ker \Pi$ is the right inverse of $L$ (cf.~\cite{GaMa-77,Ma-79,Ma-93}).

In this framework, if $\Omega\subseteq X$ is an open and bounded set such that
\begin{equation*}
Lu \neq N_{\lambda,\mu}u, \quad \text{for all $u\in \partial{\Omega}\cap \text{\rm dom}\,L$,}
\end{equation*}
the \textit{coincidence degree} $\mathrm{D}_{L}(L-N_{\lambda,\mu},\Omega)$ \textit{of $L$ and $N_{\lambda,\mu}$ in $\Omega$} is defined as
\begin{equation*}
\mathrm{D}_{L}(L-N_{\lambda,\mu},\Omega):= \text{\rm deg}_{LS}(\mathrm{Id} - \Phi_{\lambda,\mu},\Omega,0)
\end{equation*}
and it satisfies the standard properties of the topological degree, such as additivity, excision, homotopic invariance.

Our goal is to construct open and bounded sets $\Lambda\subseteq X$ such that $\mathrm{D}_{L}(L-N_{\lambda,\mu},\Lambda)\neq0$. By the existence property of the degree, this implies that there exists $u\in \Lambda\cap \text{\rm dom}\,L$ such that $Lu = N_{\lambda,\mu}u$. Therefore, $u(x)$ is a $P$-periodic solution of~\eqref{eq-flm}. To obtain a $P$-periodic solution of~$(\mathscr{E}_{\lambda,\mu})$, we further need to have
\begin{equation*}
0\leq u(x) \leq 1, \quad \text{for all $x\in\mathopen{[}0,P\mathclose{]}$.}
\end{equation*}
The first inequality follows from a simple convexity argument (the so-called maximum principle). Indeed, if $x_{0}\in\mathopen{[}0,P\mathclose{]}$ is such that $u(x_{0})=\min_{x\in\mathopen{[}0,P\mathclose{]}} u(x) <0$, then from equation \eqref{eq-flm} we obtain $u''(x)<0$ for a.e.~$x$ in a neighborhood of $x_{0}$, a contradiction.
 As for the second inequality, it will be a consequence of the construction of $\Lambda$, indeed we will take $\Lambda\subseteq \{u\in X \colon \|u\|_{\infty}<1\}$, so that $u(x)<1$ for all $x\in\mathopen{[}0,P\mathclose{]}$ (incidentally, notice that this prevents $u(x)$ to be the constant solution $u\equiv1$).

\medskip

To construct the sets $\Lambda$ as above, we need to introduce some auxiliary sets where we will compute the degree.
Given three constants $r,\rho,R$ with $0 < r < \rho < R < 1$, for any pair of
subsets of indices $\mathcal{I},\mathcal{J} \subseteq \{1,\ldots,m\}$ (possibly empty) with
$\mathcal{I} \cap \mathcal{J} = \emptyset$, we define the open and bounded set
\begin{equation*}
\Omega^{\mathcal{I},\mathcal{J}}_{(r,\rho,R)} :=
\left\{ u \in X \colon \|u\|_\infty<1,
\begin{array}{l}
\max_{I^{+}_{i}}|u|<r, \; i\in\{1,\ldots,m\}\setminus(\mathcal{I}\cup\mathcal{J})
\\
\max_{I^{+}_{i}}|u|<\rho, \; i\in\mathcal{I}
\\
\max_{I^{+}_{i}}|u|<R, \; i\in\mathcal{J}
\end{array} \right\}.
\end{equation*}
With this notation, the following lemmas hold. 

\begin{lemma}\label{lem-deg0}
Let $c\in\mathbb{R}$ and let $a \colon \mathbb{R} \to \mathbb{R}$ be a $P$-periodic locally integrable function satisfying $(a_{*})$. Let $g\colon\mathopen{[}0,1\mathclose{]}\to\mathbb{R}$ be a continuous function satisfying $(g_*)$. 
Let $\mathcal{I} \neq \emptyset$ and $\lambda, \mu >0$.
Assume that there exists $v \in L^{1}(0,P)$, with $v(x) \succ 0$ on $\mathopen{[}0,P\mathclose{]}$
and $v \equiv 0$ on $\bigcup_{i} I^{-}_{i}$, such that the following properties hold.
\begin{itemize}
\item[$(H_{1})$]
If $\alpha \geq 0$, then any $P$-periodic solution $u(x)$ of
\begin{equation}\label{eq-lem-deg0}
u'' + cu' + a_{\lambda,\mu}(x) g(u) + \alpha v(x) = 0,
\end{equation}
with $0 \leq u(x) \leq R$ for all $x\in \mathopen{[}0,P\mathclose{]}$,
satisfies
\begin{itemize}
\item[$\bullet$]
$\max_{x\in I^{+}_{i}} u(x)\neq r$, if $i \notin \mathcal{I} \cup \mathcal{J}$;
\item[$\bullet$]
$\max_{x\in I^{+}_{i}} u(x) \neq \rho$, if $i \in \mathcal{I}$;
\item[$\bullet$]
$\max_{x\in I^{+}_{i}} u(x) \neq R$, if $i \in \mathcal{J}$.
\end{itemize}
\item[$(H_{2})$]
There exists $\alpha_{0} \geq 0$ such that equation \eqref{eq-lem-deg0}, with $\alpha=\alpha_{0}$, does not possess any
non-negative $P$-periodic solution $u(x)$ with $u(x) \leq \rho$, for all $x \in \bigcup_{i \in \mathcal{I}} I^{+}_{i}$.
\end{itemize}
Then, it holds that $\mathrm{D}_{L}\bigl{(}L-N_{\lambda,\mu},\Omega^{\mathcal{I},\mathcal{J}}_{(r,\rho,R)}\bigr{)} = 0$.
\end{lemma}

\begin{lemma}\label{lem-deg1}
Let $c\in\mathbb{R}$ and let $a \colon \mathbb{R} \to \mathbb{R}$ be a $P$-periodic locally integrable function satisfying $(a_{*})$. Let $g\colon\mathopen{[}0,1\mathclose{]}\to\mathbb{R}$ be a continuous function satisfying $(g_*)$. Let $\lambda>0$ and $\mu > \mu^{\#}(\lambda)$. Assume the following property.
\begin{itemize}
\item[$(H_{3})$]
If $\vartheta\in \mathopen{]}0,1\mathclose{]}$, then any $P$-periodic solution $u(x)$ of
\begin{equation}\label{eq-lem-deg1}
u'' + cu' + \vartheta a_{\lambda,\mu}(x) g(u) = 0,
\end{equation}
with $0 \leq u(x) \leq R$ for all $x\in \mathopen{[}0,P\mathclose{]}$, satisfies
\begin{itemize}
\item[$\bullet$]
$\max_{x\in I^{+}_{i}} u(x) \neq r$, if $i \notin \mathcal{J}$;
\item[$\bullet$]
$\max_{x\in I^{+}_{i}} u(x) \neq R$, if $i \in \mathcal{J}$.
\end{itemize}
\end{itemize}
Then, it holds that $\mathrm{D}_{L}\bigl{(}L-N_{\lambda,\mu},\Omega^{\emptyset,\mathcal{J}}_{(r,\rho,R)} \bigr{)} = 1$.
\end{lemma}

The proofs of Lemma~\ref{lem-deg0} and Lemma~\ref{lem-deg1} follow the argument of the ones of~\cite[Lemma~3.1]{BoFeZa-18tams} and \cite[Lemma~3.2]{BoFeZa-18tams}, respectively (even with some simplifications, due to the fact that the sets considered in the present paper are bounded, differently from the case treated in \cite{BoFeZa-18tams}). We point out that in \cite{BoFeZa-18tams} only the case $c = 0$ was treated; however,
the presence of the term $cu'$ does not cause any additional difficulties, after having observed that the following property holds.
\begin{quote}
\textit{If $u(x)$ is a non-negative solution of either~\eqref{eq-lem-deg0} or~\eqref{eq-lem-deg1} then 
\begin{equation}\label{eq-flower}
\max_{x\in I^{-}_{i}}u(x)=\max_{x\in\partial I^{-}_{i}}u(x).
\end{equation}}
\end{quote}
When $c=0$, the above property follows straightforwardly from a convexity argument. Instead, in the present setting it can be obtained by writing equations~\eqref{eq-lem-deg0} and~\eqref{eq-lem-deg1} in the form $(e^{cx}u')'+ e^{cx} (a_{\lambda,\mu}(x) g(u)+\alpha v(x))=0$ and $(e^{cx}u')'+\vartheta e^{cx}  a_{\lambda,\mu}(x) g(u)=0$ and then arguing as in \cite[Remark~3.4]{FeZa-17jde}.

We notice that, for $d\in\mathopen{]}0,1\mathclose{[}$, by taking either $\mathcal{I}=\{1,\dots,m\}$ and $\mathcal{J}=\emptyset$ in Lemma~\ref{lem-deg0} or $\mathcal{I}=\mathcal{J}=\emptyset$ in Lemma~\ref{lem-deg1}, we can evaluate the degree on the sets of the following type 
\begin{equation*}
\big\{ u \in X \colon \|u\|_\infty<1,\, \max\nolimits_{I^{+}_{i}}|u|<d, \; i\in\{1,\ldots,m\}\big\}.
\end{equation*}
An application of property~\eqref{eq-flower} together with the excision property of the degree allows us to compute the degree on the open ball $B_{d}\subseteq X$ of center zero and radius $d>0$. More precisely, the following corollaries can be proved.

\begin{corollary}\label{cor-deg0}
Let $c\in\mathbb{R}$ and let $a \colon \mathbb{R} \to \mathbb{R}$ be a $P$-periodic locally integrable function satisfying $(a_{*})$. Let $g\colon\mathopen{[}0,1\mathclose{]}\to\mathbb{R}$ be a continuous function satisfying $(g_*)$. 
Let $\mathcal{I} \neq \emptyset$ and $\lambda, \mu >0$. Let $d\in \mathopen{]}0,1\mathclose{[}$ and assume that there exists $v \in L^{1}(0,P)$, with $v(x) \succ 0$ on $\mathopen{[}0,P\mathclose{]}$
and $v \equiv 0$ on $\bigcup_{i} I^{-}_{i}$, such that the following properties hold.
\begin{itemize}
\item[$(\widetilde{H}_{1})$]
If $\alpha \geq 0$, then any non-negative $P$-periodic solution $u(x)$ of
\eqref{eq-lem-deg0} satisfies $\|u\|_\infty\not=d$.
\item[$(\widetilde{H}_{2})$]
There exists $\alpha_{0} \geq 0$ such that equation \eqref{eq-lem-deg0}, with $\alpha=\alpha_{0}$, does not possess any
non-negative $P$-periodic solution $u(x)$ with $\|u\|_\infty\leq d$.
\end{itemize}
Then, it holds that $\mathrm{D}_{L}\bigl{(}L-N_{\lambda,\mu},B_d \bigr{)}= 0.$
\end{corollary}

\begin{corollary}\label{cor-deg1}
Let $c\in\mathbb{R}$ and let $a \colon \mathbb{R} \to \mathbb{R}$ be a $P$-periodic locally integrable function satisfying $(a_{*})$. Let $g\colon\mathopen{[}0,1\mathclose{]}\to\mathbb{R}$ be a continuous function satisfying $(g_*)$. 
Let $\lambda>0$ and $\mu > \mu^{\#}(\lambda)$. Let $d\in \mathopen{]}0,1\mathclose{[}$ and assume that the following property holds.
\begin{itemize}
\item[$(\widetilde{H}_{3})$]
If $\vartheta\in \mathopen{]}0,1\mathclose{]}$, then any non-negative $P$-periodic solution $u(x)$ of
\eqref{eq-lem-deg1} satisfies $\|u\|_\infty\not=d$.
\end{itemize}
Then, it holds that $\mathrm{D}_{L}\bigl{(}L-N_{\lambda,\mu},B_d \bigr{)} = 1$.
\end{corollary}

\subsection{Finding the constant $\lambda^{*}$}\label{section-2.2}

In the following lemma we provide the constant $\lambda^*=\lambda^*(\rho)$ that appears in all our main results.

\begin{lemma}\label{lem-rho}
Let $c\in\mathbb{R}$ and let $a \colon \mathbb{R} \to \mathbb{R}$ be a $P$-periodic locally integrable function satisfying $(a_{*})$. Let $g\colon\mathopen{[}0,1\mathclose{]}\to\mathbb{R}$ be a continuous function satisfying $(g_*)$.
Then, for every $\rho \in \mathopen{]}0,1\mathclose{[}$, there exists $\lambda^{*} = \lambda^{*}(\rho) > 0$ such that, for every $\lambda > \lambda^{*}$, $\alpha \geq 0$, and $i \in \{1,\ldots,m\}$, there are no non-negative solutions $u(x)$ of
\begin{equation}\label{eq-lem-rho}
u'' + c u' + \lambda a^{+}(x)g(u) + \alpha = 0,
\end{equation}
with $u(x)$ defined for all $x \in I^{+}_{i}$ and
such that $\max_{x\in I^{+}_{i}} u(x) = \rho$.
\end{lemma}

The proof is essentially the same as in \cite[Section~3.1]{BoFeZa-16}. However, we give the details for reader's convenience and since we need to slightly refine the estimates.

\begin{proof}  
We fix $\varepsilon > 0$ such that $\varepsilon < (\tau_{i}- \sigma_{i})/2$ and $\int_{\sigma_{i}+\varepsilon}^{\tau_{i}-\varepsilon} a^{+}(x)\,\mathrm{d}x > 0$, for every $i\in\{1,\ldots,m\}$.
Thus the quantity
\begin{equation*}
\nu_{\varepsilon} := \min_{i=1,\ldots,m} \int_{\sigma_{i}+\varepsilon}^{\tau_{i}-\varepsilon} a^{+}(x)\,\mathrm{d}x
\end{equation*}
is well defined and positive.

Let $\rho>0$ be fixed and consider $\alpha \geq 0$ and $i \in \{1,\ldots,m\}$.
Suppose that $u(x)$ is a non-negative solution of \eqref{eq-lem-rho} defined on $I^{+}_{i}=\mathopen{[}\sigma_{i},\tau_{i}\mathclose{]}$ and such that $\max_{x\in I^{+}_{i}} u(x) = \rho.$

We claim that
\begin{equation}\label{eq-3.2}
|u'(x)| \leq \dfrac{u(x)}{\varepsilon}\, e^{|c||I^{+}_{i}|}, \quad \text{for all $x\in \mathopen{[}\sigma_{i}+\varepsilon,\tau_{i}-\varepsilon\mathclose{]}$,}
\end{equation}
and that there exists~$\delta_{i} \in \mathopen{]}0,1\mathclose{[}$ (depending only on $\varepsilon$, $c$, and $|I^{+}_{i}|$) such that
\begin{equation}\label{eq-delta}
\min_{x\in \mathopen{[}\sigma_{i}+\varepsilon,\tau_{i}-\varepsilon\mathclose{]}} u(x)\geq \delta_{i} \rho.
\end{equation}

Once we prove \eqref{eq-3.2} and \eqref{eq-delta}, we can define
\begin{equation*}
\eta = \eta(\rho):= \min \bigl{\{} g(u) \colon u\in \mathopen{[}\delta_{i}\rho,\rho\mathclose{]} \bigr{\}}
\end{equation*}
and 
\begin{equation}\label{eq-lambdastar}
\lambda^{*} = \lambda^{*}(\rho) := \max_{i=1,\ldots,m} \dfrac{\rho \bigl(\varepsilon |c| + 2 e^{|c||I^{+}_{i}|}\bigl )}{\varepsilon\eta \int_{\sigma_{i}+\varepsilon}^{\tau_{i}-\varepsilon} a(x)\,\mathrm{d}x}.
\end{equation}
Then, by integrating equation~\eqref{eq-lem-rho} on $\mathopen{[}\sigma_{i}+\varepsilon,\tau_{i}-\varepsilon\mathclose{]}$ and using \eqref{eq-3.2} (for $x=\sigma_{i}+\varepsilon$ and $x= \tau_{i}-\varepsilon$), we obtain
\begin{equation*}
\begin{aligned}
\lambda \eta \int_{\sigma_{i}+\varepsilon}^{\tau_{i}-\varepsilon} a(x)\,\mathrm{d}x
   &\leq \lambda \int_{\sigma_{i}+\varepsilon}^{\tau_{i}-\varepsilon} a(x) g(u(x))\,\mathrm{d}x
\\ &= u'(\sigma_{i}+\varepsilon) - u'(\tau_{i}-\varepsilon) + c \bigl{(}u(\sigma_{i}+\varepsilon) - u(\tau_{i}-\varepsilon)\bigr{)} \\&\quad- \alpha \, (\tau_{i}-\varepsilon-\sigma_{i}-\varepsilon)
\\ &\leq 2 \dfrac{\rho}{\varepsilon} e^{|c||I^{+}_{i}|} + |c| \rho.
\end{aligned}
\end{equation*}
Therefore, non-negative $P$-periodic solutions $u(x)$ of \eqref{eq-lem-rho} with
$\max_{x\in I} u(x) = \rho$ can exist only for $\lambda\leq\lambda^*$. This proves the lemma.

\smallskip
\noindent\textit{Proving estimate~\eqref{eq-3.2}}. Since~\eqref{eq-lem-rho} is equivalent to $(e^{cx} u')' + e^{cx} ( \lambda a^{+}(x)g(u)+\alpha ) = 0$, we observe that the map $x\mapsto e^{cx} u'(x)$ is non-increasing on $I^{+}_{i}$. Let us fix $x\in \mathopen{[}\sigma_{i}+\varepsilon,\tau_{i}-\varepsilon\mathclose{]}$. If $u'(x)=0$ then the estimates is obvious. Otherwise, from $u'(x)>0$ and by using the monotonicity of the map $x\mapsto e^{cx} u'(x)$, we have that 
\begin{equation*}
u'(\xi)\geq u'(x) e^{c(x-\xi)}, \quad \text{for all $\xi\in\mathopen{[}\sigma_{i},x\mathclose{]}$.}
\end{equation*}
By integrating the above inequality we obtain 
\begin{equation*}
u(x)\geq u(x)-u(\sigma_{i})\geq u'(x)e^{-|c|(x-\sigma_{i})}(x-\sigma_{i})\geq \varepsilon u'(x)  e^{-|c||I^{+}_{i}|}
\end{equation*}
that implies~\eqref{eq-3.2}. The case $u'(x)<0$ is analogous.

\smallskip
\noindent\textit{Proving estimate~\eqref{eq-delta}}. Let $x_{0}\in I^{+}_{i}$ be such that $u(x_{0}) = \rho$ and observe that $u'(x_{0}) = 0$, if $\sigma_{i} < x_{0} < \tau_{i}$, while $u'(x_{0}) \leq 0$, if $x_{0} =\sigma_{i}$, and $u'(x_{0}) \geq 0$, if $x_{0} = \tau_{i}$. If $x_{*}\in \mathopen{[}\sigma_{i}+\varepsilon,\tau_{i}-\varepsilon\mathclose{]}$
is such that $u(x_{*}) = \min_{x\in \mathopen{[}\sigma_{i}+\varepsilon,\tau_{i}-\varepsilon\mathclose{]}} u(x)$, from~\eqref{eq-3.2} we obtain that
\begin{equation}\label{eq-3.4}
|u'(x_{*})| \leq \dfrac{u(x_{*})}{\varepsilon} \, e^{|c||I^{+}_{i}|}.
\end{equation}
On the other hand, by the monotonicity of the function $x\mapsto e^{cx} u'(x)$ in $\mathopen{[}\sigma_{i},\tau_{i}\mathclose{]}$,
\begin{equation}\label{eq-3.5}
u'(\xi) e^{c\xi} \geq u'(x_{*}) e^{cx_{*}}, \quad \text{for all $\xi\in \mathopen{[}\sigma_{i},x_{*}\mathclose{]}$,}
\end{equation}
and
\begin{equation}\label{eq-3.6}
u'(\xi) e^{c\xi} \leq u'(x_{*}) e^{cx_{*}}, \quad \text{for all $\xi\in \mathopen{[}x_{*},\tau_{i}\mathclose{]}$.}
\end{equation}
From the properties about $u'(x_{0})$, we have that if $x_{0} > x_{*}$,
then $u'(x_{0}) \geq 0$ and so $u'(x_{*}) \geq 0$.
Similarly, if $x_{0} < x_{*}$, then $u'(x_{0}) \leq 0$ and so $u'(x_{*}) \leq 0$.
The case $x_{*}=x_{0}$ is trivial.
As a consequence we have either
\begin{equation}\label{eq-3.7}
\sigma_{i} \leq x_{0} < x_{*}\in \mathopen{[}\sigma_{i}+\varepsilon,\tau_{i}-\varepsilon\mathclose{]},
\;\;  u(x_{0}) = \rho, \;\; u'(\xi) \leq 0, \;\; \text{for all $\xi\in \mathopen{[}x_{0},x_{*}\mathclose{]}$,}
\end{equation}
or
\begin{equation}\label{eq-3.8}
\tau_{i}\geq x_{0} > x_{*}\in \mathopen{[}\sigma_{i}+\varepsilon,\tau_{i}-\varepsilon\mathclose{]},
\;\;  u(x_{0}) = \rho, \;\; u'(\xi) \geq 0, \;\; \text{for all $\xi\in \mathopen{[}x_{*},x_{0}\mathclose{]}$.}
\end{equation}
When~\eqref{eq-3.7} holds, from \eqref{eq-3.5} we have $-u'(\xi) \leq - u'(x_{*}) e^{c(x_{*}-\xi)}$ for all $\xi\in \mathopen{[}x_{0},x_{*}\mathclose{]}$. An integration of the previous inequality on $\mathopen{[}x_{0},x_{*}\mathclose{]}$ and an application of \eqref{eq-3.4} lead to \begin{equation*}
\rho - u(x_{*}) \leq |u'(x_{*})| \, e^{|c||I^{+}_{i}|}(x_{*}-x_{0}) \leq \dfrac{u(x_{*})}{\varepsilon} \, e^{2|c||I^{+}_{i}|}|I^{+}_{i}|.
\end{equation*}
Then, by fixing
\begin{equation*}
\delta_{i}:= \dfrac{\varepsilon}{\varepsilon + e^{2|c||I^{+}_{i}|} |I^{+}_{i}|}
\end{equation*}
we have \eqref{eq-delta}. The same estimate follows in the case of \eqref{eq-3.8}, by using \eqref{eq-3.6} and integrating on $\mathopen{[}x_{*},x_{0}\mathclose{]}$.
\end{proof}

\subsection{Some estimates for small solutions}\label{section-2.3}

The following lemma gives a lower bound for positive $T$-periodic solutions of~\eqref{eq-lem-deg1} that will be exploited in the proof of the existence result in Theorem~\ref{th-exis}.

\begin{lemma}\label{lem-r0}
Let $c\in\mathbb{R}$ and let $a \colon \mathbb{R} \to \mathbb{R}$ be a $P$-periodic locally integrable function satisfying $(a_{*})$. Let $g\colon\mathopen{[}0,1\mathclose{]}\to\mathbb{R}$ be a continuously differentiable function satisfying $(g_*)$ and $g'(0)=0.$ Let $\lambda>0$ and $\mu>\mu^{\#}(\lambda).$ Then, there exists $r_{0}\in\mathopen{]}0,1\mathclose{[}$ such that for every $\vartheta\in\mathopen{]}0,1\mathclose{]}$, every non-negative $P$-periodic solution $u(x)$ of~\eqref{eq-lem-deg1} with $\|u\|_\infty\leq r_{0}$ satisfies $u\equiv 0$.
\end{lemma}

\begin{proof}
Let $M>e^{|c|P} \|a\|_{L^{1}(0,P)}$.
By contradiction, we assume that there exists a sequence $(u_{n}(x))_{n}$ of non-negative $P$-periodic solutions of~\eqref{eq-lem-deg1} for $\vartheta=\vartheta_{n}\in\mathopen{]}0,1\mathclose{]}$ satisfying $0<\|u_{n}\|_{\infty}\to0$. We perform the change of variable 
\begin{equation}\label{change-var}
z_{n}(x) := \dfrac{u_{n}'(x)}{\vartheta_{n} g(u_{n}(x))}, \quad x\in\mathbb{R}.
\end{equation}
An easy computation shows that
\begin{equation}\label{eq-zn}
z_{n}'(x)+c z_{n}(x) + \vartheta_{n} g'(u_{n}(x))z_{n}^{2}(x) +a_{\lambda,\mu}(x) = 0.
\end{equation}
We claim that
\begin{equation*}
\|z_{n}\|_{\infty} \leq M.
\end{equation*}
We suppose by contradiction that this is not true. Then, recalling the fact that $z_{n}(x)$ vanishes at some point $\tilde{x}_{n}\in\mathopen{[}0,P\mathclose{]}$, we can find a maximal interval $J_{n}\subseteq \mathopen{[}0,P\mathclose{]}$ either of the form $\mathopen{[}\tilde{x}_{n},\hat{x}_{n}\mathclose{]}$ or of the form $\mathopen{[}\hat{x}_{n},\tilde{x}_{n}\mathclose{]}$, such that $|z_{n}(x)| \leq M$ for all $x\in J_{n}$ and $|z_{n}(x)| > M$ for some $x\notin J_{n}$. By the maximality of the interval $J_{n}$, we also know that $|z_{n}(\hat{x}_{n})| = M$.
Rewriting \eqref{eq-zn} as
\begin{equation*}
\bigl{(} e^{c(x-\hat{x}_{n})}z_{n}(x) \bigr{)}' + e^{c(x-\hat{x}_{n})} \bigl{(} \vartheta_{n} g'(u_{n}(x))z_{n}^{2}(x) +a_{\lambda,\mu}(x)\bigr{)} = 0,
\end{equation*}
an integration on $J_{n}$ gives
\begin{equation*}
z_{n}(\hat{x}_{n}) = - \int_{J_{n}} \bigl{(} e^{c(x-\hat{x}_{n})} \bigl{(} \vartheta_{n} g'(u_{n}(x))z_{n}^{2}(x) +a_{\lambda,\mu}(x)\bigr{)}  \bigr{)} \, \mathrm{d}x 
\end{equation*}
from which
\begin{equation*}
M = |z_{n}(\hat{x}_{n})| \leq e^{|c|P} \biggl{(} \sup_{x \in \mathopen{[}0,P\mathclose{]}} |g '(u_{n}(x))| PM^{2}  +\|a_{\lambda,\mu}\|_{L^{1}(0,P)} \biggr{)}.
\end{equation*}
Passing to the limit as $n\to\infty$ and using $g'(0)=0$ we thus obtain $M\leq e^{|c|P}\|a_{\lambda,\mu}\|_{L^{1}(0,P)}$, contradicting the choice of $M$.

Now, we integrate \eqref{eq-zn} on $\mathopen{[}0,P\mathclose{]}$ to obtain
\begin{equation*}
0 < -\int_{0}^{P} a_{\lambda,\mu}(x)\,\mathrm{d}x \leq \sup_{x \in \mathopen{[}0,P\mathclose{]}} |g'(u_{n}(x))| PM^{2},
\end{equation*}
and so a contradiction is reached using the fact that $g'(u)$ is continuous and $g'(0)=0$.
\end{proof}

The next lemma gives us some estimates for positive solutions of~\eqref{eq-lem-deg1} which will be used to prove the multiplicity result in Theorem~\ref{th-mult}. To state it, let us introduce the following notation. For any constant $d > 0$, we set
\begin{equation*}
\zeta(d) := \max_{\frac{d}{2}\leq u \leq d} \dfrac{g(u)}{u}, \qquad
\gamma(d) := \min_{\frac{d}{2}\leq u \leq d} \dfrac{g(u)}{u}.
\end{equation*}
Furthermore, recalling $(a_{*})$ and the positions in \eqref{Ipm},
for all $i\in\{1,\ldots, m\}$, we set
\begin{align}\label{eq-AB}
A^{\mathrm{r}}_{i}(x) := \int_{\tau_{i}}^{x} a^{-}(\xi)\,\mathrm{d}\xi, \quad A^{\mathrm{l}}_{i}(x) := \int_{x}^{\sigma_{i+1}} a^{-}(\xi)\,\mathrm{d}\xi, \quad x\in I^{-}_{i}.
\end{align}

\begin{lemma}\label{lem-r}
Let $c\in\mathbb{R}$ and let $a \colon \mathbb{R} \to \mathbb{R}$ be a $P$-periodic locally integrable function satisfying $(a_{*})$. Let $g\colon\mathopen{[}0,1\mathclose{]}\to\mathbb{R}$ be a continuous function satisfying $(g_*)$ and $(g_{0})$. Let $\lambda>0$. Then, there exists $\bar{r}\in\mathopen{]}0,1\mathclose{[}$ such that for every $r\in\mathopen{]}0,\bar{r}\mathclose{]}$, for every $\vartheta\in\mathopen{]}0,1\mathclose{]}$, and for every $\mu>0$, if $u(x)$ is a non-negative solution of~\eqref{eq-lem-deg1} defined in $I^{-}_{i-1}\cup I^{+}_{i} \cup I^{-}_{i}$ for some $i\in \{1,\ldots,m\}$ with $\max_{x\in I^{+}_{i}} u(x) = r,$ the following hold:
\begin{itemize}
\item if $u'(\sigma_{i}) \geq 0$, then
\begin{alignat*}{1}
&u(\sigma_{i+1}) \geq r \biggl{(} 1 + \dfrac{\vartheta}{2} \Bigl{(}  \mu \gamma(r) \|A^{\mathrm{r}}_{i}\|_{L^1(I_{i}^{-})} e^{-|c||I^{-}_{i}|} - 1 \Bigr{)} \biggr{)},
\shortintertext{and}
&u'(\sigma_{i+1}) \geq \vartheta r \biggl{(} \dfrac{1}{2} \mu \gamma(r) \|a\|_{L^{1}(I^{-}_{i})}e^{-|c||I^{-}_{i}|}  \\
&\hspace{2.3cm}-\lambda \|a\|_{L^{1}(I^{+}_{i})} \zeta(r) e^{|c||I^{+}_{i}\cup I^{-}_{i}|}\biggr{)};
\end{alignat*}
\item if $u'(\tau_{i}) \leq 0$, then
\begin{alignat*}{1}
&u(\tau_{i-1}) \geq r \biggl{(} 1 + \dfrac{\vartheta}{2} \Bigl{(}  \mu \gamma(r) \|A^{\mathrm{l}}_{i-1}\|_{L^1(I_{i-1}^{-})} e^{-|c||I^{-}_{i-1}|} - 1 \Bigr{)} \biggr{)},
\shortintertext{and}
&u'(\tau_{i-1}) \leq -\vartheta r \biggl{(} \dfrac{1}{2} \mu \gamma(r) \|a\|_{L^{1}(I^{-}_{i-1})}e^{-|c||I^{-}_{i-1}|}  \\
&\hspace{2.5cm}- \lambda \|a\|_{L^{1}(I^{+}_{i})} \zeta(r) e^{|c||I^{-}_{i-1}\cup I^{+}_{i}|}\biggr{)}.
\end{alignat*}
\end{itemize}
\end{lemma}

\begin{proof}
From condition $(g_{0})$ we can fix a constant $\bar{r}\in\mathopen{]}0,1\mathclose{[}$ such that for every $r\in\mathopen{]}0,\bar{r}\mathclose{]}$ it holds that
\begin{equation}\label{cond-d}
\zeta(r) < \dfrac{1}{2 \lambda \, \displaystyle{\max_{i=1,\ldots,m}}e^{|c||I^{-}_{i-1}\cup I^{+}_{i}\cup I^{-}_{i}|} |I^{-}_{i-1}\cup I^{+}_{i}\cup I^{-}_{i}|\|a\|_{L^1(I^{+}_{i})}}.
\end{equation}

We give the proof when $u'(\sigma_{i}) \geq 0$ (the case $u'(\tau_{i}) \leq 0$ follows from analogous arguments). 
We divide the arguments into two parts: in the first one, we provide some estimates for $u(\tau_{i})$ and $u'(\tau_{i})$, in the second one, we obtain the inequalities on $u(\sigma_{i+1})$ and $u'(\sigma_{i+1})$.

\noindent
\textit{Step~1.} Let $\hat{x}_{i}\in I^{+}_{i}$ be such that
\begin{equation*}
u(\hat{x}_{i})=\max_{t\in I^{+}_{i}} u(x) = r.
\end{equation*}
We notice that if $\sigma_{i} \leq \hat{x}_{i} < \tau_{i}$, then $u'(\hat{x}_{i}) = 0$ (since $u'(\sigma_{i}) \geq 0$). Otherwise, if $\hat{x}_{i} = \tau_{i}$, then $u'(\hat{x}_{i}) \geq 0$.

Suppose first that $u'(\hat{x}_{i}) = 0$. Let $\mathopen{[}s_{1},s_{2}\mathclose{]} \subseteq I^{+}_{i}$ be the maximal closed interval containing $\hat{x}_{i}$ and such that
$u(x)\geq r/2$ for all $x\in\mathopen{[}s_{1},s_{2}\mathclose{]}$. We claim that $\mathopen{[}s_{1},s_{2}\mathclose{]} = I^{+}_{i}$.
From
\begin{equation*}
\bigl{(}e^{cx} u'(x)\bigr{)} ' = -\vartheta \lambda a^{+}(x) g(u(x)) e^{cx}, \quad x \in I^{+}_{i}, 
\end{equation*}
integrating between $\hat{x}_{i}$ and $x$ and using $u'(\hat{x}_{i}) = 0$, we obtain
\begin{equation*}
u'(x) = -\vartheta \lambda \int_{\hat{x}_{i}}^{x} a^{+}(\xi) g(u(\xi)) e^{c(\xi-x)} \,\mathrm{d}\xi, \quad \text{for all $x\in I^{+}_{i}$.}
\end{equation*}
Then,
\begin{equation*}
|u'(x)| \leq \vartheta \lambda \|a\|_{L^{1}(I^{+}_{i})} \zeta(r)r e^{|c||I^{+}_{i}|}, \quad \text{for all $x\in \mathopen{[}s_{1},s_{2}\mathclose{]}$,}
\end{equation*}
and
\begin{align*}
u(x) &= u(\hat{x}_{i}) + \int_{\hat{x}_{i}}^{x} u'(\xi)\,\mathrm{d}\xi  \\
&\geq r \Bigl{(}1 - \lambda \|a\|_{L^{1}(I^{+}_{i})} \zeta(r) e^{|c||I^{+}_{i}|}|I^{+}_{i}| \Bigr{)}>\dfrac{r}{2}, \quad \text{for all $x\in \mathopen{[}s_{1},s_{2}\mathclose{]}$.}
\end{align*}
This inequality, together with the maximality of $\mathopen{[}s_{1},s_{2}\mathclose{]}$, implies that $\mathopen{[}s_{1},s_{2}\mathclose{]}=I^{+}_{i}$.
Hence
\begin{equation}\label{eq-xx}
u'(x) \geq -\vartheta \lambda \|a\|_{L^{1}(I^{+}_{i})} \zeta(r)r e^{|c||I^{+}_{i}|}, \quad \text{for all $x\in I^{+}_{i}$,}
\end{equation}
implying
\begin{equation}\label{eq-4.2tau}
u'(\tau_{i}) \geq - \vartheta\lambda \|a\|_{L^{1}(I^{+}_{i})} \zeta(r)r e^{|c||I^{+}_{i}|}.
\end{equation}
Furthermore, by integrating \eqref{eq-xx} on $\mathopen{[}\hat{x}_{i},\tau_{i}\mathclose{]}$, we obtain
\begin{equation}\label{eq-4.3tau}
u(\tau_{i}) \geq r \bigl{(} 1 - \vartheta \lambda  \|a\|_{L^{1}(I^{+}_{i})} \zeta(r) e^{|c||I^{+}_{i}|}|I^{+}_{i}|\bigr{)}.
\end{equation}

On the other hand, if we suppose that
$\hat{x}_{i} = \tau_{i}$ and $u'(\hat{x}_{i})> 0$, we have
\begin{equation*}
u(\tau_{i}) = r \geq r \bigl{(} 1 - \vartheta \lambda  \|a\|_{L^{1}(I^{+}_{i})} \zeta(r) e^{|c||I^{+}_{i}|}|I^{+}_{i}|\bigr{)} 
\end{equation*}
and
\begin{equation*}
u'(\tau_{i}) >0 \geq  - \vartheta\lambda \|a\|_{L^{1}(I^{+}_{i})} \zeta(r)r e^{|c||I^{+}_{i}|}.
\end{equation*}
Thus, in any case, \eqref{eq-4.2tau} and \eqref{eq-4.3tau} hold, and so we can proceed with the second part of the proof.

\smallskip

\noindent
\textit{Step~2.} We consider the interval $I^{-}_{i}=\mathopen{[}\tau_{i},\sigma_{i+1}\mathclose{]}$.
Since the map $x\mapsto e^{cx} u'(x)$ is non-decreasing in $I^{-}_{i}$, from \eqref{eq-4.2tau} we have 
\begin{equation*}
u'(x) \geq e^{c(\tau_{i}-x)}u'(\tau_{i}) \geq - \vartheta \lambda \|a\|_{L^{1}(I^{+}_{i})} \zeta(r) r e^{|c||I^{+}_{i}\cup I^{-}_{i}|}, \quad \text{for all $x\in I^{-}_{i}$.}
\end{equation*}
Therefore, integrating on $\mathopen{[}\tau_{i},x\mathclose{]}$ and using \eqref{eq-4.3tau}, we have
\begin{align}\label{eq-4.9}
u(x) &= u(\tau_{i}) + \int_{\tau_{i}}^{x} u'(\xi) \,\mathrm{d}\xi\\
&\geq r \Bigl{(} 1 - \vartheta \lambda  \|a\|_{L^{1}(I^{+}_{i})} \zeta(r) e^{|c||I^{+}_{i}|}|I^{+}_{i}|- \vartheta \lambda \|a\|_{L^{1}(I^{+}_{i})} \zeta(r) e^{|c||I^{+}_{i}\cup I^{-}_{i}|}|I^{-}_{i}|
\Bigr{)} \\
&\geq r \Bigl{(}1- \lambda |I^{+}_{i} \cup I^{-}_{i}|\|a\|_{L^{1}(I^{+}_{i})} \zeta(r) e^{|c||I^{+}_{i}\cup I^{-}_{i}|}	 \Bigr{)}\\
&\geq r \Bigl{(}1- \lambda |I^{-}_{i-1}\cup I^{+}_{i} \cup I^{-}_{i}|\|a\|_{L^{1}(I^{+}_{i})} \zeta(r) e^{|c||I^{-}_{i-1} \cup I^{+}_{i}\cup I^{-}_{i}|}	 \Bigr{)}\\
&>\dfrac{r}{2}, \quad \text{for all $ x\in I^{-}_{i}$,}
\end{align}
where the last inequality follows from \eqref{cond-d}. On the other hand, integrating
\begin{equation*}
\bigl{(}e^{cx}u'(x)\bigr{)}' = \vartheta \mu a^{-}(x) g(u(x)) e^{cx}, \quad x \in I^{-}_{i},
\end{equation*}
on $\mathopen{[}\tau_{i},x\mathclose{]}$ and using \eqref{eq-4.2tau} and \eqref{eq-4.9}, we find
\begin{equation*}
\begin{aligned}
u'(x)
   &= u'(\tau_{i}) e^{c(\tau_{i}-x)} + \vartheta \mu\int_{\tau_{i}}^{x} a^{-}(\xi) g(u(\xi)) e^{c(\xi-x)}\,\mathrm{d}\xi
\\ &\geq \vartheta r \biggl{(}-\lambda \|a\|_{L^{1}(I^{+}_{i})} \zeta(r) e^{|c||I^{+}_{i}\cup I^{-}_{i}|}  + \dfrac{1}{2} \mu \gamma(r) A^{\mathrm{r}}_{i}(x)e^{-|c||I^{-}_{i}|} \biggr{)}, \quad \text{for all $x \in I^{-}_{i}$.}
\end{aligned}
\end{equation*}
In particular,
\begin{equation*}
u'(\sigma_{i+1}) \geq \vartheta r \biggl{(}\dfrac{1}{2} \mu \gamma(r) \|a\|_{L^{1}(I^{-}_{i})}e^{-|c||I^{-}_{i}|}  - \lambda \|a\|_{L^{1}(I^{+}_{i})} \zeta(r) e^{|c||I^{+}_{i}\cup I^{-}_{i}|}\biggr{)}.
\end{equation*}
Finally, a further integration and condition \eqref{eq-4.3tau} provide
\begin{equation*}
\begin{aligned}
u(\sigma_{i+1})
   &= u(\tau_{i}) + \int_{\tau_{i}}^{\sigma_{i+1}} u'(x) \,\mathrm{d}x
\\ &\geq r \biggl{(} 1 - \vartheta \lambda  \|a\|_{L^{1}(I^{+}_{i})} \zeta(r)e^{|c||I^{+}_{i}|}|I^{+}_{i}| 
\\ &\qquad -\vartheta \lambda \|a\|_{L^{1}(I^{+}_{i})} \zeta(r) e^{|c||I^{+}_{i}\cup I^{-}_{i}|}|I^{-}_{i}|+ \vartheta \dfrac{1}{2} \mu \gamma(r) \|A^{\mathrm{r}}_{i}\|_{L^{1}(I^{-}_{i})} e^{-|c||I^{-}_{i}|} \biggr{)}
\\ &\geq r \biggl{(} 1 - \vartheta \lambda  \|a\|_{L^{1}(I^{+}_{i})} \zeta(r)e^{|c||I^{-}_{i-1}\cup I^{+}_{i}\cup I^{-}_{i}|}|I^{-}_{i-1}\cup I^{+}_{i}\cup I^{-}_{i}|
\\ &\qquad+ \vartheta \dfrac{1}{2} \mu \gamma(r) \|A^{\mathrm{r}}_{i}\|_{L^{1}(I^{-}_{i})} e^{-|c||I^{-}_{i}|} \biggr{)}
\\ &\geq r \biggl{(} 1 + \dfrac{\vartheta}{2} \Bigl{(}  \mu \gamma(r) \|A^{\mathrm{r}}_{i}\|_{L^{1}(I^{-}_{i})} e^{-|c||I^{-}_{i}|} - 1 \Bigr{)} \biggr{)},
\end{aligned}
\end{equation*}
where the last inequality follows from~\eqref{cond-d}. Thus the proof is completed.
\end{proof}

\subsection{Some estimates for large solutions}\label{section-2.4}

We start by introducing the following auxiliary result.

\begin{lemma}\label{lem-Re}
 Let $c\in\mathbb{R}$. Let $g\colon\mathopen{[}0,1\mathclose{]}\to\mathbb{R}$ be a continuous function satisfying $(g_*)$ and $(g_{1})$. Let $J\subseteq \mathbb{R}$ be a closed interval and $b\in L^1(J)$. Then, for every $\varepsilon\in\mathopen{]}0,1\mathclose{[}$ there exists $R_{\varepsilon}=R_{\varepsilon}(c,g,J,b)\in\mathopen{]}0,1\mathclose{[}$ such that for every $\vartheta \in \mathopen{]}0,1\mathclose{]}$ and for every non-negative solution $u(x)$ of
\begin{equation*}
u'' + cu' +\vartheta b(x) g(u) = 0,
\end{equation*}
that satisfies $u(\hat{x})\geq R_\varepsilon$ and $u'(\hat{x})=0$ for some $\hat{x}\in J$, it holds that  
\begin{equation*}
u(x)\geq1-\varepsilon \quad\text{and}\quad |u'(x)| \leq \varepsilon, \quad \text{for all $x\in J$}. 
\end{equation*}
\end{lemma}

\begin{proof}
Given $\varepsilon\in\mathopen{]}0,1\mathclose{[}$, let us define 
\begin{equation*}
R_\varepsilon =1- {\varepsilon}{e^{-\frac{1}{2}{K\|b\|_{L^1(J)}+(1+2|c|)|J|}}}.
\end{equation*}
First of all we notice that either $u\equiv 1$ or $(1-u(x))^{2}+(u'(x))^{2}>0$ for every $x\in J$, due to the uniqueness of the solution of the Cauchy problem
\begin{equation*}
\begin{cases}
\, u'' + cu' +\vartheta b(x) g(u) = 0,\\
\, u(x_{0})=1,\\
\, u'(x_{0})=0,
\end{cases}
\end{equation*}
ensured by condition $(g_{1})$.

In the first case the thesis follows straightforwardly. In the second case, we compute 
\begin{align*}
&\frac{\mathrm{d}}{\mathrm{d}x} \log{\bigl{(}(1-u(x))^{2}+(u'(x))^{2}\bigr{)}} =\\
&=-2\;\frac{(1-u(x))u'(x)+\vartheta b(x) u'(x)g(u(x))+c (u'(x))^{2}}{(1-u(x))^{2}+(u'(x))^{2}}.
\end{align*}
From the previous equality and since by $(g_{1})$ we can fix $K>0$ such that $g(u)\leq K(1-u)$ for every $u\in\mathopen{[}0,1\mathclose{]}$, we deduce that 
\begin{align*}
&\biggl{|}\frac{\mathrm{d}}{\mathrm{d}x} \log{\bigl{(}(1-u(x))^{2}+(u'(x))^{2}\bigr{)}}\biggr{|} \leq\\
&\leq2\frac{(1-u(x))|u'(x)|+|b(x)| |u'(x)|g(u(x))+|c| (u'(x))^{2}}{(1-u(x))^{2}+(u'(x))^{2}}\\
&\leq 2\frac{(1+K|b(x)|)(1-u(x))|u'(x)|+|c| (u'(x))^{2}}{(1-u(x))^{2}+(u'(x))^{2}}\\
&\leq 1+K|b(x)|+2|c|.
\end{align*}
Hence, by an integration of the above inequality from $\hat{x}$ to an arbitrary $x\in J$, we have 
\begin{equation*}
\log{\frac{(1-u(x))^{2} +(u'(x))^{2}}{(1-u(\hat{x}))^{2} }}\leq  K\|b\|_{L^{1}(J)}+(1+2|c|)|J|.
\end{equation*}
As a consequence, it follows that 
\begin{equation*}
(1-u(x))^{2} +(u'(x))^{2}\leq (1-R_\varepsilon)^{2}e^{K\|b\|_{L^{1}(J)}+(1+2|c|)|J|}= \varepsilon^{2},
\end{equation*}
for all $x\in J$, and so the thesis is proved.
\end{proof}

The following lemma gives an upper bound for positive $P$-periodic solutions of~\eqref{eq-lem-deg1} which will be used to prove the existence result in Theorem~\ref{th-exis}.

\begin{lemma}\label{lem-R0}
Let $c\in\mathbb{R}$ and let $a \colon \mathbb{R} \to \mathbb{R}$ be a $P$-periodic locally integrable function satisfying $(a_{*})$. Let $g\colon\mathopen{[}0,1\mathclose{]}\to\mathbb{R}$ be a continuously differentiable function satisfying $(g_*)$. Let $\lambda>0$ and $\mu>\mu^{\#}(\lambda).$ Then, there exists $R_{0}\in\mathopen{]}0,1\mathclose{[}$ such that for every $\vartheta\in \mathopen{]}0,1\mathclose{]}$, every non-negative $P$-periodic solution $u(x)$ of~\eqref{eq-lem-deg1} satisfies $\|u\|_\infty< R_{0}$. \end{lemma}

\begin{proof}
By contradiction we assume that there exists a sequence $(u_{n}(x))_{n}$ of non-negative $P$-periodic solutions of \eqref{eq-lem-deg1} for $\vartheta=\vartheta_{n}\in\mathopen{]}0,1\mathclose{]}$ such that $\|u_{n}\|_{\infty}\to1^{-}$. 

By applying Lemma~\ref{lem-Re} with the choice of $J=\mathopen{[}0,P\mathclose{]}$ and $b(x)=a_{\lambda,\mu}(x)$, we deduce that $u_{n}(x)\to1$ uniformly in $x$ as $n\to\infty$. 

Through the change of variable introduced in \eqref{change-var} and an integration of \eqref{eq-zn} on $\mathopen{[}0,P\mathclose{]}$ we have
\begin{equation}\label{eq-cc}
0> \int_{0}^{P} a_{\lambda,\mu}(x) \,\mathrm{d}x=-\vartheta_{n} \int_{0}^{P} g'(u_{n}(x))z_{n}^{2}(x)\,\mathrm{d}x.
\end{equation}
When $g'(1)<0$ we deduce that $g'(u)<0$ for every $u$ in a left neighborhood of~$1$. In this case, a contradiction follows from \eqref{eq-cc} by the uniform convergence of $u_{n}(x)$ to~$1$. When $g'(1)=0$, a contradiction is reached because, by arguing as in Lemma~\ref{lem-r0}, the sequence $(z_{n}(x))_{n}$ is uniformly bounded and $g'(u_{n}(x))$ converges to~$0$ uniformly.
\end{proof}

The next lemma gives us some estimates for positive solutions of~\eqref{eq-lem-deg1} which will be used to prove the multiplicity result in Theorem~\ref{th-mult}. To state it, we recall the definition of $A^{\mathrm{r}}_{i}(x)$ and $A^{\mathrm{l}}_{i}(x)$ given in \eqref{eq-AB} and we introduce the further notation
\begin{equation*}
\Gamma(d) := \max_{0 \leq u \leq d} g(u), \qquad
\chi(d,D) := \min_{d \leq u \leq D} g(u),
\end{equation*}
where $d,D\in\mathopen{]}0,1\mathclose{[}$ satisfy $d<D$.

\begin{lemma}\label{lem-R}
Let $c\in\mathbb{R}$ and let $a \colon \mathbb{R} \to \mathbb{R}$ be a $P$-periodic locally integrable function satisfying $(a_{*})$. Let $g\colon\mathopen{[}0,1\mathclose{]}\to\mathbb{R}$ be a continuous function satisfying $(g_*)$ and $(g_{1})$. Let $\lambda>0$ and $d\in\mathopen{]}0,1\mathclose{[}$. Then, there exists $\bar{R}=\bar{R}(d)\in\mathopen{]}d,1\mathclose{[}$ such that for every $R\in\mathopen{[}\bar{R},1\mathclose{[}$, $\vartheta\in\mathopen{]}0,1\mathclose{]}$ and $\mu>0$, if $u(x)$ is a non-negative solution of~\eqref{eq-lem-deg1} defined in $I^{-}_{i-1}\cup I^{+}_{i} \cup I^{-}_{i}$ for some $i\in \{1,\ldots,m\}$ with $\max_{x\in I^{-}_{i-1}\cup I^{+}_{i} \cup I^{-}_{i}} u(x) = \max_{x\in I^{+}_{i}} u(x) = R$ it holds that
\begin{alignat*}{1}
&u(\sigma_{i+1}) \geq R + \vartheta \biggl{(} \mu  \|A^{\mathrm{r}}_{i}\|_{L^{1}(I^{-}_{i})} \chi(d,R) e^{-|c||I^{-}_{i}|} 
\\ & \hspace{100pt} -\lambda  \|a\|_{L^{1}(I^{+}_{i})} \Gamma(R)e^{|c||I^{+}_{i}\cup I^{-}_{i}|}|I^{+}_{i}\cup I^{-}_{i}| \biggr{)}
\shortintertext{and}
&u(\tau_{i-1}) \geq R + \vartheta \biggl{(} \mu  \|A^{\mathrm{l}}_{i-1}\|_{L^{1}(I^{-}_{i-1})} \chi(d,R) e^{-|c||I^{-}_{i-1}|} 
\\ & \hspace{100pt} -\lambda  \|a\|_{L^{1}(I^{+}_{i})} \Gamma(R)e^{|c||I^{-}_{i-1}\cup I^{+}_{i}|}||I^{-}_{i-1}\cup I^{+}_{i}| \biggr{)}.
\end{alignat*}
\end{lemma}

\begin{proof}
Given $d>0$, let us take
\begin{equation*}
\varepsilon=\frac{1-d}{\displaystyle 1+\max_{i=1,\ldots,m} {|I^{-}_{i}| e^{|c| |I^{-}_{i}|}}}.
\end{equation*}
We apply Lemma~\ref{lem-Re} with the choice of $J=I^{+}_{i}$ and $b(x)=\lambda a^{+}(x)$ in order to find the corresponding $R_{\varepsilon,i}=R_{\varepsilon}(c,g,I^{+}_{i},\lambda a^{+})$ and we set
\begin{equation*}
\bar{R} = \bar{R}(d) = \max_{i=1,\ldots,m} R_{\varepsilon,i}.
\end{equation*}
Notice that $1-\varepsilon > d$. Therefore, since $R_{\varepsilon,i}\in\mathopen{]}1-\varepsilon,\varepsilon\mathclose{[}$, it holds that $\bar{R}\in\mathopen{]}d,1\mathclose{[}$.

Let $R\in\mathopen{[}\bar{R},1\mathclose{[}$, $\vartheta\in\mathopen{]}0,1\mathclose{]}$ and $\mu>0$. Let $u(x)$ be a non-negative solution of~\eqref{eq-lem-deg1} defined in $I^{-}_{i-1}\cup I^{+}_{i} \cup I^{-}_{i}$ for some $i\in \{1,\ldots,m\}$ with 
\begin{equation*}
\max_{x\in I^{-}_{i-1}\cup I^{+}_{i} \cup I^{-}_{i}} u(x) = \max_{t\in I^{+}_{i}} u(x) = R.
\end{equation*}

Let $\hat{x}_{i}\in I^{+}_{i}$ be such that $u(\hat{x}_{i})=\max_{x\in I^{+}_{i}} u(x) = R$. We observe that $u'(\hat{x}_{i}) = 0$, otherwise $u(x)>R$ for some $x$ in a neighborhood of $\hat{x}_{i}$.
Lemma~\ref{lem-Re} applies and yields
\begin{equation}\label{eq-epsR}
u(x)\geq1-\varepsilon \quad\text{and}\quad |u'(x)| \leq \varepsilon, \quad \text{for all $x\in I^{+}_{i}$}. 
\end{equation}
We claim that
\begin{equation*}
u(x)\geq d, \quad \text{for all $x\in I^{-}_{i-1}\cup I^{+}_{i} \cup I^{-}_{i}$.}
\end{equation*}
The inequality in $I^{+}_{i}$ is obvious since 
$1-\varepsilon>d$. As for the interval $I^{-}_{i}$, since the map $x\mapsto e^{cx} u'(x)$ is non-decreasing, we have $e^{cx} u'(x) \geq e^{c\tau_{i}} u'(\tau_{i})$, for all $x\in I^{-}_{i}$.
Thus, from \eqref{eq-epsR} it follows that
\begin{equation*}
|u'(x)| \leq \varepsilon e^{|c||I^{-}_{i}|}, \quad \text{for all $x\in I^{-}_{i}$.}
\end{equation*}
Then, an integration gives
\begin{equation*}
u(x) = u(\tau_{i}) + \int_{\tau_{i}}^{x} u'(\xi) \,\mathrm{d}\xi \geq 1-\varepsilon-\varepsilon |I^{-}_{i}| e^{|c||I^{-}_{i}|} \geq d, \quad \text{for all $x\in I^{-}_{i}$,}
\end{equation*}
where the last inequality follows from the choice of $\varepsilon$.
A similar argument applies in the interval $I^{-}_{i-1}$ and the claim is thus proved.

Recalling that $u(\hat{x}_{i}) = 0$, we find
\begin{equation*}
u'(x) = -\vartheta \lambda \int_{\hat{x}_{i}}^{x} a^{+}(\xi) g(u(\xi)) e^{c(\xi-x)} \,\mathrm{d}\xi, \quad \text{for all $x\in I^{+}_{i}$,}
\end{equation*}
implying
\begin{equation*}
|u'(x)| \leq \vartheta \lambda \|a\|_{L^{1}(I^{+}_{i})} \Gamma(R) e^{|c||I^{+}_{i}|}, \quad \text{for all $x\in I^{+}_{i}$.}
\end{equation*}
Therefore
\begin{equation*}
u(\tau_{i}) = u(\hat{x}_{i}) + \int_{\hat{x}_{i}}^{\tau_{i}} u'(\xi)\,\mathrm{d}\xi  \geq R - \vartheta \lambda \|a\|_{L^{1}(I^{+}_{i})} \Gamma(R) e^{|c||I^{+}_{i}|}|I^{+}_{i}|.
\end{equation*}
As a consequence, in the interval $I^{-}_{i}$ we have
\begin{equation*}
\begin{aligned}
u'(x)
   &= u'(\tau_{i}) e^{c(\tau_{i}-x)} + \vartheta \mu\int_{\tau_{i}}^{x} a^{-}(\xi) g(u(\xi)) e^{c(\xi-x)}\,\mathrm{d}\xi
\\ &\geq - \vartheta \lambda \|a\|_{L^{1}(I^{+}_{i})} \Gamma(R) e^{|c||I^{+}_{i}\cup I^{-}_{i}|}  + \vartheta \mu A^{\mathrm{r}}_{i}(x)\chi(d,R) e^{-|c||I^{-}_{i}|}, \quad \text{for all $x \in I^{-}_{i}$.}
\end{aligned}
\end{equation*}
An integration of the above inequality, together with the estimate for $u(\tau_{i})$, finally provides
\begin{equation*}
\begin{aligned}
u(\sigma_{i+1})
   &= u(\tau_{i}) + \int_{\tau_{i}}^{\sigma_{i+1}} u'(x) \,\mathrm{d}x
\\ &\geq R - \vartheta \lambda  \|a\|_{L^{1}(I^{+}_{i})} \Gamma(R) e^{|c||I^{+}_{i}|}|I^{+}_{i}| 
\\ &\quad -\vartheta \lambda \|a\|_{L^{1}(I^{+}_{i})} \Gamma(R) e^{|c||I^{+}_{i}\cup I^{-}_{i}|}|I^{-}_{i}|+ \vartheta \mu \|A^{\mathrm{r}}_{i}\|_{L^{1}(I^{-}_{i})} \chi(d,R) e^{-|c||I^{-}_{i}|}
\\ &\geq R + \vartheta \biggl{(} \mu  \|A^{\mathrm{r}}_{i}\|_{L^{1}(I^{-}_{i})} \chi(d,R) e^{-|c||I^{-}_{i}|} 
\\ & \qquad -\lambda  \|a\|_{L^{1}(I^{+}_{i})} \Gamma(R)e^{|c||I^{+}_{i}\cup I^{-}_{i}|}|I^{+}_{i}\cup I^{-}_{i}| \biggr{)}
\end{aligned}
\end{equation*}
where the last inequality follows from~\eqref{cond-d}. Thus the proof is completed.
\end{proof}

\begin{remark}\label{rem-2.1}
Lemma~\ref{lem-R} will be exploited in Section~\ref{section-4.1}, while verifying the assumptions of Lemma~\ref{lem-deg0} and 
Lemma~\ref{lem-deg1}. We stress that only the assertion on $u(\sigma_{i+1})$ will be used. The second one plays a role in the corresponding proofs dealing with Dirichlet or Neumann boundary conditions (see Section~\ref{section-6.2}).
$\hfill\lhd$
\end{remark}

\section{Existence of two solutions}\label{section-3}

In this section we give the proof of Theorem~\ref{th-exis}.

\begin{proof}[Proof of Theorem~\ref{th-exis}.]
Given $\rho>0$, we first apply Lemma~\ref{lem-rho} in order to find the constant $\lambda^{*} =\lambda^{*} (\rho) > 0$ (defined as in \eqref{eq-lambdastar}). Then we fix $\lambda > \lambda^{*}$. 

We claim that Corollary~\ref{cor-deg0} applies with the choice of $d=\rho$ and $v(x)$ as the indicator function $\mathbbm{1}_{\bigcup_{i} I^{+}_{i}}(x)$ of the set $\bigcup_{i} I^{+}_{i}$, that is,
\begin{equation*}
v(x) =
\begin{cases}
\, 1, & \text{if $x \in \bigcup_{i = 1}^m I^{+}_{i}$,} \\
\, 0, & \text{if $x \in \mathopen{[}0,P\mathclose{]} \setminus \bigcup_{i=1}^m I^{+}_{i}$.}
\end{cases}
\end{equation*}
First, we verify assumption $(\widetilde{H}_{1})$. From property \eqref{eq-flower}, since $v(x) = 0$ for all $x \in \bigcup_{i} I^{-}_{i}$, we observe that any non-negative $P$-periodic solution of \eqref{eq-lem-deg0} attains its maximum on $\bigcup_{i} I^{+}_{i}$. Then, $(\widetilde{H}_{1})$ follows from Lemma~\ref{lem-rho}. 
As for assumption $(\widetilde{H}_{2})$, we integrate equation \eqref{eq-lem-deg0} on $\mathopen{[}0,P\mathclose{]}$ and pass to the absolute value in order to obtain
\begin{equation*}
\alpha \| v \|_{L^{1}(0,P)} \leq \| a_{\lambda,\mu} \|_{L^{1}(0,P)} \max_{u \in \mathopen{[}0,\rho\mathclose{]}} g(u).
\end{equation*}
Therefore, $(\widetilde{H}_{2})$ follows for $\alpha$ sufficiently large. Summing up, from Corollary~\ref{lem-deg0}, we thus obtain
\begin{equation*}
\mathrm{D}_{L}(L-N_{\lambda,\mu},B_{\rho}) = 0.
\end{equation*}

Now, we use Lemma~\ref{lem-r0} and Lemma~\ref{lem-R0} to fix $r_{0}$ and $R_{0}$ in $ \mathopen{]}0,1\mathclose{[}$. Without loss of generality we can assume $0<r_{0} < \rho < R_{0} < 1$. Then, Corollary~\ref{cor-deg1} applies both with the choice of $d=r_{0}$ and $d=R_{0}$ (indeed, $(\widetilde{H}_{3})$ is trivially satisfied). Therefore, we have
\begin{equation*}
\mathrm{D}_{L}(L-N_{\lambda,\mu},B_{r_{0}}) = 1 \quad \text{ and } \quad 
\mathrm{D}_{L}(L-N_{\lambda,\mu},B_{R_{0}}) = 1.
\end{equation*}

The additivity property of the coincidence degree implies
\begin{equation*}
\mathrm{D}_{L}(L-N_{\lambda,\mu},B_{\rho} \setminus \overline{B_{r_{0}}}) = -1
\quad \text{ and } \quad \mathrm{D}_{L}(L-N_{\lambda,\mu},B_{R_{0}} \setminus \overline{B_{\rho}}) = 1.
\end{equation*}
As a consequence, there exist a $P$-periodic solution $u_{s}(x)$ of \eqref{eq-flm} in $B_{\rho} \setminus \overline{B_{r_{0}}}$ as well as a $P$-periodic solution $u_{\ell}(x)$ of \eqref{eq-flm} in $B_{R_{0}} \setminus \overline{B_{\rho}}$. As observed in Section~\ref{section-2.1}, by the maximum principle it holds that $u_{s}(x) \geq 0$ and $u_{\ell}(x) \geq 0$ for all $x\in\mathopen{[}0,P\mathclose{]}$. Moreover, we clearly have $u_{s}(x) <1$ and $u_{\ell}(x) <1$ for all $x\in\mathopen{[}0,P\mathclose{]}$. Hence, $u_{s}(x)$ and $u_{\ell}(x)$ are non-negative $P$-periodic solutions of $(\mathscr{E}_{\lambda,\mu})$.
Since $g(u)$ is of class $\mathcal{C}^1$, the uniqueness of the constant zero solution for the Cauchy problem associated with $(\mathscr{E}_{\lambda,\mu})$, implies that $u_{s}(x)$ and $u_{\ell}(x)$ are positive $P$-periodic solutions of $(\mathscr{E}_{\lambda,\mu})$ and the proof is concluded.
\end{proof}

\begin{remark}
By a careful checking of the proof, one can realize that Theorem~\ref{th-exis} is still valid if $g(u)$ is assumed to be continuously differentiable in a right neighborhood of $u = 0$ and in a left neighborhood of $u=1$. We also remark that the assumption of 
differentiability near $u = 0$ could be removed, provided one supposes a condition of regular oscillation, that is,
\begin{equation*}
\lim_{\substack{u \to 0^{+} \\ \omega \to 1}} \frac{g(\omega u)}{g(u)} = 1
\end{equation*}
(cf.~\cite[Section~4.3]{BoFeZa-16}). At last, we mention that, by arguing as in \cite{BoFeZa-16}, one could also weaken assumption $(a_*)$, so as to cover some situations when the weight function $a(x)$ changes sign infinitely many times. For the sake of briefness,
and since assumption $(a_*)$ is crucial in the proof of Theorem~\ref{th-mult}, we have preferred to work in a unified simpler setting.
$\hfill\lhd$
\end{remark}

We end this section by stating the following straightforward corollary, dealing with the one-parameter equation \eqref{eq-lambda}.

\begin{corollary}\label{cor-exis}
Let $c\in\mathbb{R}$ and let $a \colon \mathbb{R} \to \mathbb{R}$ be a $P$-periodic locally integrable function satisfying $(a_{*})$ and $\int_{0}^{P} a(x) \,\mathrm{d}x <0$. Let $g \colon \mathopen{[}0,1\mathclose{]} \to \mathbb{R}$ be a continuously differentiable function satisfying $(g_{*})$ and $(g_{0})$.
Then, there exists $\lambda^{*} > 0$ (depending on $c$, $g(u)$ and $a^{+}(x)$, but not on $a^{-}(x)$) such that for every $\lambda > \lambda^{*}$ equation
\eqref{eq-lambda} has at least two non-constant positive $P$-periodic solutions.
\end{corollary}

\section{High multiplicity of solutions}\label{section-4}

In this section we give the proof of Theorem~\ref{th-mult}. 

\begin{proof}[Proof of Theorem~\ref{th-mult}.]
Given $\rho>0$, we first apply Lemma~\ref{lem-rho} in order to find the constant $\lambda^{*} =\lambda^{*} (\rho) > 0$ (defined as in \eqref{eq-lambdastar}). Then we fix $\lambda > \lambda^{*}$. 

We apply Lemma~\ref{lem-r} to find $\bar{r}\in\mathopen{]}0,1\mathclose{[}$ and we fix 
\begin{equation*}
r\in\mathopen{]}0,\min\{\bar{r},\rho\}\mathclose{[}.
\end{equation*}
Moreover, we apply Lemma~\ref{lem-R}, with the choice of $d=\rho$, to find $\bar{R}\in\mathopen{]}\rho,1\mathclose{[}$  and we fix
\begin{equation*}
R\in\mathopen{[}\bar{R},1\mathclose{[}.
\end{equation*}

We claim that there exists $\mu^{*}(\lambda) =\mu^{*}(\lambda,r,R)>0$ such that for every $\mu>\mu^{*}(\lambda)$ Lemma~\ref{lem-deg0} and Lemma~\ref{lem-deg1} hold for any pair of
subsets of indices $\mathcal{I},\mathcal{J} \subseteq \{1,\ldots,m\}$ with
$\mathcal{I} \cap \mathcal{J} = \emptyset$. This is a long technical step of the proof and we provide the details in Section~\ref{section-4.1}.
Once this is proved, we have that 
\begin{equation}\label{eq-3.1}
\mathrm{D}_{L}\bigl{(}L-N_{\lambda,\mu},\Omega^{\mathcal{I},\mathcal{J}}_{(r,\rho,R)} \bigr{)} =
\begin{cases}
\, 0, & \text{if } \;\mathcal{I} \neq \emptyset, \\
\, 1, & \text{if } \;\mathcal{I} = \emptyset.
\end{cases}
\end{equation}

We define the open and bounded sets
\begin{equation*}
\Lambda^{\mathcal{I},\mathcal{J}}_{(r,\rho,R)} :=
\left\{ u \in X \colon \|u\|_\infty<1,
\begin{array}{l}
 \max_{I^{+}_{i}}|u|<r, \; i\in\{1,\ldots,m\}\setminus(\mathcal{I}\cup\mathcal{J})
\\
r<\max_{I^{+}_{i}}|u|<\rho, \; i\in\mathcal{I}
\\
\rho<\max_{I^{+}_{i}}|u|<R, \; i\in\mathcal{J}
\end{array} \right\}
\end{equation*}
and so from \eqref{eq-3.1} and the combinatorial argument in \cite[Appendix~A]{BoFeZa-18tams}, we obtain that
\begin{equation*}
\mathrm{D}_{L}\bigl{(}L-N_{\lambda,\mu},\Lambda^{\mathcal{I},\mathcal{J}}_{(r,\rho,R)}\bigr{)} =
(-1)^{\# \mathcal{I}}.
\end{equation*}

As a consequence of the existence property for the coincidence degree, we thus obtain the existence of a $P$-periodic solution of \eqref{eq-flm} in each of these $3^{m}$ sets $\Lambda^{\mathcal{I},\mathcal{J}}_{(r,\rho,R)}$. Here, the number $3^{m}$ comes from all the possible choices $\mathcal{I}$ and $\mathcal{J}$ with $\mathcal{I} \cap \mathcal{J} = \emptyset$.
Notice that, since the identically zero function is contained in the set $\Lambda^{\emptyset,\emptyset}_{(r,\rho,R)}$, we do not consider it in the sequel.
Instead, every solution $u(x)$ of \eqref{eq-flm} in each of the other $3^{m}-1$ sets is non-constant and, by the maximum principle, such that $u(x)\geq 0$ for all $x\in\mathopen{[}0,P\mathclose{]}$. By the uniqueness of the zero solution for the Cauchy problem associated with \eqref{eq-flm} (coming from condition $(g_{0})$) we have also $u(x)> 0$ for all $x\in\mathopen{[}0,P\mathclose{]}$. Moreover, by construction, it follows that $u(x)<1$ for all $x\in\mathopen{[}0,P\mathclose{]}$. Hence, $u(x)$ is a non-constant positive $P$-periodic solution of $(\mathscr{E}_{\lambda,\mu})$.

Summing up, for each choice of $\mathcal{I}$ and $\mathcal{J}$
with $\mathcal{I} \cap \mathcal{J} = \emptyset \neq \mathcal{I} \cup \mathcal{J}$, there exists at least one positive $P$-periodic solution $u_{\mathcal{I},\mathcal{J}}(x)$ of $(\mathscr{E}_{\lambda,\mu})$ such that
\begin{itemize}
\item $0 < \max_{x\in I^{+}_{i}} u_{\mathcal{I},\mathcal{J}}(x) < r$, for all $i \notin \mathcal{I} \cup \mathcal{J}$;
\item $r < \max_{x\in I^{+}_{i}} u_{\mathcal{I},\mathcal{J}}(x) < \rho$, for all $i \in \mathcal{I}$;
\item $\rho < \max_{x\in I^{+}_{i}} u_{\mathcal{I},\mathcal{J}}(x) < R$, for all $i \in \mathcal{J}$.
\end{itemize}
Finally, to achieve the conclusion of Theorem~\ref{th-mult}, we observe that,
given any finite string $\mathcal{S} = (\mathcal{S}_{1},\ldots,\mathcal{S}_{m}) \in \{0,1,2\}^{m}$, with $\mathcal{S} \neq (0,\ldots,0)$,
we can establish a one-to-one correspondence between $\mathcal{S}$ and the sets
\begin{equation*}
{\mathcal{I}}:= \bigl{\{} i\in \{1,\ldots,m\} \colon \mathcal{S}_{i} =1 \bigr{\}},\quad
{\mathcal{J}}:= \bigl{\{} i\in \{1,\ldots,m\} \colon \mathcal{S}_{i} =2 \bigr{\}},
\end{equation*}
so that $\mathcal{S}_{i} = 0$ when $i\notin {\mathcal{I}} \cup {\mathcal{J}}$.
This completes the proof of Theorem~\ref{th-mult}.
\end{proof}

\subsection{Finding the constant $\mu^{*}(\lambda,r,R)$}\label{section-4.1}

The constant $\mu^{*}(\lambda,r,R)$ is defined as
\begin{equation*}
\mu^{*}(\lambda,r,R) :=\max \Bigl{\{} \mu^{(H_{1})},\mu^{(H_{3})} \Bigr{\}},
\end{equation*}
where $\mu^{(H_{1})}$ and $\mu^{(H_{3})}$ will be obtained along the arguments below (see \eqref{def-H1} and \eqref{def-H3}). We stress that such constants are fully explicit, depending only on $\lambda$, $r$, $\rho$,  $R$, $g(u)$ and $a(x)$.

\subsubsection*{Checking the assumptions of Lemma~\ref{lem-deg0}.}
Let $\mathcal{I},\mathcal{J}$ with $\mathcal{I}\neq\emptyset$ and define $v(x)$ as the indicator function of the set $\bigcup_{i\in\mathcal{I}}I^{+}_{i}$, namely
\begin{equation*}
v(x) =
\begin{cases}
\, 1, & \text{if $x \in \bigcup_{i\in\mathcal{I}}I^{+}_{i}$,} \\
\, 0, & \text{if $x \in \mathopen{[}0,P\mathclose{]} \setminus \bigcup_{i\in\mathcal{I}}I^{+}_{i}$.}
\end{cases}
\end{equation*}

\smallskip
\noindent
\textit{Verification of $(H_{1})$.}
Let $\alpha \geq 0$. By contradiction, we suppose that there exists a $P$-periodic solution $u(x)$ of \eqref{eq-lem-deg0}
with $0\leq u(x) \leq R$, for all $x\in\mathopen{[}0,P\mathclose{]}$, such that at least one of the following conditions holds:
\begin{itemize}
\item[$(h_{1}^{1})$] there is an index $i \notin \mathcal{I}\cup\mathcal{J}$ such that $\max_{x\in I^{+}_{i}} u(x) = r$;
\item[$(h_{2}^{1})$] there is an index $i \in \mathcal{I}$ such that $\max_{x\in I^{+}_{i}} u(x) = \rho$;
\item[$(h_{3}^{1})$] there is an index $i \in \mathcal{J}$ such that $\max_{x\in I^{+}_{i}} u(x) = R$.
\end{itemize}

Suppose that $(h_{1}^{1})$ holds. Since $v(x)=0$ for $x\in I^{-}_{i-1} \cup I^{+}_{i}\cup I^{-}_{i}$, equation \eqref{eq-lem-deg0} reduces to $(\mathscr{E}_{\lambda,\mu})$. 
Consider at first the case $u'(\sigma_{i}) \geq 0$.
By Lemma~\ref{lem-r} (with $\vartheta=1$), we have that
\begin{align*}
u(\sigma_{i+1}) 
&\geq r \biggl{(} 1 + \dfrac{1}{2} \Bigl{(}  \mu \gamma(r) \|A^{\mathrm{r}}_{i}\|_{L^1(I_{i}^{-})} e^{-|c||I^{-}_{i}|} - 1 \Bigr{)} \biggr{)} \\
&\geq \dfrac{\mu}{2} r \gamma(r)\|A^{\mathrm{r}}_{i}\|_{L^1(I_{i}^{-})} e^{-|c||I^{-}_{i}|}.
\end{align*}
Thus, taking
\begin{equation}\label{muright}
\mu > \hat{\mu}_{i}^{\text{\rm r}}:= \dfrac{2R e^{|c||I^{-}_{i}|}}{r\gamma(r) \|A^{\mathrm{r}}_{i}\|_{L^1(I_{i}^{-})}},
\end{equation}
we obtain $u(\sigma_{i+1}) > R$, a contradiction.
On the other hand, if $u'(\sigma_{i})< 0$, using the fact that $x\mapsto e^{cx}u'(x)$ is non-increasing on $I^{+}_{i}$, we have that $u'(\tau_{i})< 0$. In this case, we can use the second part of Lemma~\ref{lem-r} (with $\vartheta=1$) to reach the contradiction $u(\tau_{i-1}) > R$ whenever
\begin{equation}\label{muleft}
\mu > \hat{\mu}_{i}^{\text{\rm l}}:= \dfrac{2R e^{|c||I^{-}_{i-1}|}}{r\gamma(r) \|A^{\mathrm{l}}_{i-1}\|_{L^{1}(I^{-}_{i-1})}}.
\end{equation}

Now, we suppose that $(h_{2}^{1})$ holds. In this case a contradiction is immediately obtained by Lemma~\ref{lem-rho} (no assumption on $\mu > 0$ is needed).

At last, we assume that $(h_{3}^{1})$ holds. As for the case $(h_{1}^{1})$ we have $v(x)=0$ for $x\in I^{-}_{i-1} \cup I^{+}_{i}\cup I^{-}_{i}$. Then we can apply Lemma~\ref{lem-R} (with $d=\rho$ and $\vartheta=1$) in order to obtain
\begin{align*}
u(\sigma_{i+1}) &\geq R + \mu  \|A^{\mathrm{r}}_{i}\|_{L^1(I_{i}^{-})} \chi(\rho,R) e^{-|c||I^{-}_{i}|} 
\\ & \quad -\lambda  \|a\|_{L^{1}(I^{+}_{i})} \Gamma(R)e^{|c||I^{+}_{i}\cup I^{-}_{i}|}|I^{+}_{i}\cup I^{-}_{i}|.
\end{align*}
Taking
\begin{equation*}
\mu > \check{\mu}_{i}^{\text{\rm r}}:= \dfrac{\lambda  \|a\|_{L^{1}(I^{+}_{i})} \Gamma(R)e^{|c||I^{+}_{i}\cup I^{-}_{i}|}|I^{+}_{i}\cup I^{-}_{i}|}{\|A^{\mathrm{r}}_{i}\|_{L^1(I_{i}^{-})} \chi(\rho,R) e^{-|c||I^{-}_{i}|} },
\end{equation*}
we obtain $u(\sigma_{i+1}) > R$, a contradiction. Notice that, contrarily to the case $(h_{1}^1)$, here it is not necessary to consider the behavior of $u(x)$ in the interval $I_{i-1}^{-}$.

We conclude that $(H_{1})$ holds for
\begin{equation}\label{def-H1}
\mu > \mu^{(H_{1})} := \max_{i=1,\ldots,m} \bigl{\{} \hat{\mu}_{i}^{\text{\rm r}},\hat{\mu}_{i}^{\text{\rm l}},\check{\mu}_{i}^{\text{\rm r}}\bigr{\}}.
\end{equation}

\smallskip

\noindent
\textit{Verification of $(H_{2})$.}
Let $u(x)$ be an arbitrary non-negative $P$-periodic solution of \eqref{eq-lem-deg0} such
that $u(x) \leq \rho$ for all $x\in\bigcup_{i\in\mathcal{I}} I^{+}_{i}$.
We fix an index $j\in\mathcal{I}$ and observe that on the interval $I^{+}_{j}$ equation \eqref{eq-lem-deg0} reads as
\begin{equation*}
u'' + cu' +\lambda a^{+}(x) g(u) + \alpha = 0.
\end{equation*}
Let $\varepsilon \in \mathopen{]}0, (\tau_{j}- \sigma_{j})/2\mathclose{[}$. As shown along the proof of Lemma~\ref{lem-rho} the inequality \eqref{eq-3.2} holds. Then, integrating the differential equation on $\mathopen{[}\sigma_{j}+\varepsilon,\tau_{j}-\varepsilon\mathclose{]}$, we obtain
\begin{align*}
&\alpha \, (\tau_{j} - \sigma_{j} - 2\varepsilon) =\\
&= u'(\sigma_{j}+\varepsilon) - u'(\tau_{j}-\varepsilon) + cu(\sigma_{j}+\varepsilon) - cu(\tau_{j}-\varepsilon)-
\lambda \int_{\sigma_{j}+\varepsilon}^{\tau_{j}-\varepsilon} a^{+}(x)g(u(x))\,\mathrm{d}x 
\\ &\leq \dfrac{2\rho}{\varepsilon} e^{|c||I^{+}_{j}|} + 2|c|\rho.
\end{align*}
This yields a contradiction if $\alpha > 0$ is sufficiently large. Hence $(H_{2})$ is verified.
\qed

\subsubsection*{Checking the assumptions of Lemma~\ref{lem-deg1}.}

Let $\mathcal{J}\subseteq\{1,\ldots,m\}$ and $\vartheta\in \mathopen{]}0,1\mathclose{]}$.

\smallskip

\noindent
\textit{Verification of $(H_{3})$.}
By contradiction, suppose that there exists a $P$-periodic solution $u(x)$ of
\eqref{eq-lem-deg1} with $0 \leq u(x) \leq R$ for all $x\in \mathopen{[}0,P\mathclose{]}$, such that at least one of the following conditions holds:
\begin{itemize}
\item[$(h_{1}^{3})$] there is an index $i \notin \mathcal{J}$ such that $\max_{x\in I^{+}_{i}} u(x) = r$;
\item[$(h_{2}^{3})$] there is an index $i \in \mathcal{J}$ such that $\max_{x\in I^{+}_{i}} u(x) = R$.
\end{itemize}

Suppose that $(h_{1}^{3})$ holds. 
We consider at first the case $u'(\sigma_{i})\geq0$. We are going to prove that, if $\mu$ large enough, then
\begin{equation}\label{eq-uru'0}
u(x) > r \quad \text{ and } \quad u'(x)>0,
\end{equation}
for all $x\in \mathopen{[}0,P\mathclose{]}$. This clearly contradicts the $P$-periodicity of $u(x)$.

\noindent
\textit{Proving \eqref{eq-uru'0} in $I_{i+1}^{+}$.}
Taking $\mu > \hat{\mu}^{\text{\rm r}}_{i}$ (with $\hat{\mu}^{\text{\rm r}}_{i}$ defined in \eqref{muright}) then we have
\begin{equation}
\mu>\dfrac{e^{|c||I^{-}_{i}|}}{\gamma(r)\|A^{\mathrm{r}}_{i}\|_{L^1(I_{i}^{-})} }
\end{equation}
and so, from Lemma~\ref{lem-r}, $u(\sigma_{i+1})>r$ (as $\vartheta > 0$).
Moreover, using the estimate on $u'(\sigma_{i+1})$ provided in Lemma~\ref{lem-r}, we observe that $u'(\sigma_{i+1})>0$ when
\begin{equation}\label{eq-mu1}
\mu > \dfrac{2 \lambda \|a\|_{L^{1}(I^{+}_{i})} \zeta(r) e^{2|c||I^{+}_{i}\cup I^{-}_{i}|}}{ \gamma(r) \|a\|_{L^{1}(I^{-}_{i})}}.
\end{equation}
Integrating \eqref{eq-lem-deg1} on $\mathopen{[}\sigma_{i+1},x\mathclose{]} \subseteq I^{+}_{i+1}$ and using again Lemma~\ref{lem-r}, we obtain
\begin{equation*}
\begin{aligned}
u'(x) &= u'(\sigma_{i+1}) e^{c(\sigma_{i+1}-x)} - \vartheta \lambda \int_{\sigma_{i+1}}^{x} a^{+}(\xi) g(u(\xi))e^{c(\xi-x)} \,\mathrm{d}\xi 
\\ &\geq u'(\sigma_{i+1}) e^{-|c||I^{+}_{i+1}|} - \vartheta\lambda \|a\|_{L^{1}(I^{+}_{i+1})}\Gamma(R) e^{|c||I^{+}_{i+1}|}
\\ &\geq \vartheta r \biggl{(}\dfrac{1}{2} \mu \gamma(r) \|a\|_{L^{1}(I^{-}_{i})}e^{-|c||I^{-}_{i}\cup I^{+}_{i+1}|}  - \lambda \|a\|_{L^{1}(I^{+}_{i})} \zeta(r) e^{|c||I^{+}_{i}\cup I^{-}_{i} \cup I^{+}_{i+1}|}
\\ &
- \lambda \|a\|_{L^{1}(I^{+}_{i+1})}\dfrac{\Gamma(R)}{r} e^{|c||I^{+}_{i+1}|}
\biggr{)}.
\end{aligned}
\end{equation*}
Notice that the first of the above inequalities requires $u'(\sigma_{i+1})\geq0$, which is ensured by \eqref{eq-mu1}.
Taking
\begin{equation*}
\mu > \tilde{\mu}^{\text{\rm r}}_{i}:=
\dfrac{2\lambda \bigl{(} \|a\|_{L^{1}(I^{+}_{i})} \zeta(r) r e^{|c||I^{+}_{i}\cup I^{-}_{i} \cup I^{+}_{i+1}|}
+ \|a\|_{L^{1}(I^{+}_{i+1})} \Gamma(R) e^{|c||I^{+}_{i+1}| } \bigr{)}}{ \gamma(r)r \|a\|_{L^{1}(I^{-}_{i})}e^{-|c||I^{-}_{i}\cup I^{+}_{i+1}|}},
\end{equation*}
we finally obtain that
\begin{equation*}
u'(x)>0, \quad \text{for all $x\in I^{+}_{i+1}$.}
\end{equation*}
Consequently $u(x) \geq u(\sigma_{i+1}) > r$ on $I^{+}_{i+1}$.
We conclude that for 
\begin{equation*}
\mu > \max\bigl{\{}\hat{\mu}^{\text{\rm r}}_{i},\tilde{\mu}^{\text{\rm r}}_{i}\bigr{\}},
\end{equation*}
inequalities in \eqref{eq-uru'0} hold.

\noindent
\textit{Proving \eqref{eq-uru'0} in $I_{i+1}^{-}$.}
Using the monotonicity of the map $x\mapsto e^{cx}u'(x)$ we deduce that
$u'(x)\geq e^{c(\tau_{i+1}-x)}u'(\tau_{i+1})>0$ on $I^{-}_{i+1}$. Thus the conclusion follows, since $u(\tau_{i+1}) > r$.

\noindent
\textit{Proving \eqref{eq-uru'0} in $I_{i+2}^{+}$.}
Integrating the equation \eqref{eq-lem-deg1} on $\mathopen{[}\tau_{i+1},x\mathclose{]} \subseteq I^{-}_{i+1}$ we find
\begin{align*}
u'(x) &= u'(\tau_{i+1})e^{c(\tau_{i+1}-x)} + \vartheta \mu \int_{\tau_{i+1}}^{x} a^{-}(\xi) g(u(\xi))e^{c(\xi-x)} \,\mathrm{d}\xi 
\\ &> \vartheta \mu A^{\mathrm{r}}_{i+1}(x) \chi(r,R) e^{-|c||I^{-}_{i+1}|}, \quad \text{for all $x\in I^{-}_{i+1}$,}
\end{align*}
in particular
\begin{equation*}
u'(\sigma_{i+2}) > \vartheta \mu \|a\|_{L^{1}(I^{-}_{i+1})}\, \chi(r,R) e^{-|c||I^{-}_{i+1}|} >0.
\end{equation*}
On the other hand, integrating the equation \eqref{eq-lem-deg1} on $\mathopen{[}\sigma_{i+2},x\mathclose{]}\subseteq I^{+}_{i+2}$ we find
\begin{equation*}
\begin{aligned}
u'(x)
   &= u'(\sigma_{i+2}) e^{c(\sigma_{i+2}-x)} -\vartheta \lambda \int_{\sigma_{i+2}}^{x} a^{+}(\xi) g(u(\xi)) e^{c(\xi-x)}\,\mathrm{d}\xi
\\ &> \vartheta \Bigl{(} \mu \|a\|_{L^{1}(I^{-}_{i+1})} \chi(r,R)e^{-|c||I^{-}_{i+1}\cup I^{+}_{i+2}|} - \lambda \|a\|_{L^{1}(I^{+}_{i+2})} \Gamma(R) e^{|c||I^{+}_{i+2}|} \Bigr{)}
> 0,
\end{aligned}
\end{equation*}
for all $ x\in I^{+}_{i+2}$, 
where the last inequality holds for
\begin{equation*}
\mu > \mu^{*,+}_{i} = \mu^{*,+}_{i}(I^{-}_{i+1},I^{+}_{i+2}) := \dfrac{\lambda \|a\|_{L^{1}(I^{+}_{i+2})} \Gamma(R) e^{2|c||I^{-}_{i+1}\cup I^{+}_{i+2}|}}{\|a\|_{L^{1}(I^{-}_{i+1})}\chi(r,R)}.
\end{equation*}
Then the solution $u(x)$ is increasing in $I^{+}_{i+2}$ and hence $u(x) > u(\sigma_{i+2}) > r$ on $I^{+}_{i+2}$.
Therefore, the inequalities in \eqref{eq-uru'0} hold in $I^{+}_{i+2}$.

\noindent
\textit{Proving \eqref{eq-uru'0} in $\mathopen{[}0,P\mathclose{]}$.}
This is easily achieved by repeating the argument just described in order to cover a $P$-periodicity interval. This eventually requires
\begin{equation*}
\mu > \max_{i=1,\ldots,m} \mu_{i}^{*,+}.
\end{equation*}

Having dealt with the case $u'(\sigma_{i}) \geq 0$, we now assume $u'(\sigma_{i}) < 0$, which implies (by the monotonicity of the map $x \mapsto e^{cx}u'(x)$ in $I^{+}_{i}$) that
$u'(\tau_{i})< 0$. A contradiction can be achieved proceeding backward. More precisely, we may use at first Lemma~\ref{lem-r} and then an inductive argument similar to the one explained above.
Conditions on $\mu$ will be replaced by the analogous inequalities
\begin{equation*}
\mu > \hat{\mu}^{\text{\rm l}}_{i},
\end{equation*}
with $\hat{\mu}^{\text{\rm l}}_{i}$ defined in \eqref{muleft},
\begin{equation*}
\mu > \tilde{\mu}^{\text{\rm l}}_{i}:=
\dfrac{2\lambda \bigl{(} \|a\|_{L^{1}(I^{+}_{i})} \zeta(r) r e^{|c||I^{+}_{i-1}\cup I^{-}_{i-1} \cup I^{+}_{i}|}
+ \|a\|_{L^{1}(I^{+}_{i-1})} \Gamma(R) e^{|c||I^{+}_{i-1}| } \bigr{)}}{ \gamma(r)r \|a\|_{L^{1}(I^{-}_{i-1})}e^{-|c||I^{-}_{i-1}\cup I^{+}_{i-1}|}},
\end{equation*}
and
\begin{equation*}
\mu > \mu^{*,-}_{i} = \mu^{*,-}_{i}(I^{+}_{i-2},I^{-}_{i-2}) := \dfrac{\lambda \|a\|_{L^{1}(I^{+}_{i-2})} \Gamma(R) e^{2|c||I^{+}_{i-2}\cup I^{-}_{i-2}|}}{\|a\|_{L^{1}(I^{-}_{i-2})}\chi(r,R)}.
\end{equation*}
Thus the contradiction $u'(x) < 0$ for all $x \in \mathopen{[}0,P\mathclose{]}$ can be proved for
\begin{equation*}
\mu > \max_{i=1,\ldots,m} \mu_{i}^{*,-}.
\end{equation*}

Taking into account all the possible situations we conclude that the case $(h_{1}^{3})$ never occurs if
\begin{equation*}
\mu > \mu_{1}^{(H_{3})} : = \max_{i=1,\ldots,m} \bigl{\{} \hat{\mu}^{\mathrm{r}}_{i},
\hat{\mu}^{\mathrm{l}}_{i}, \tilde{\mu}^{\mathrm{r}}_{i}, \tilde{\mu}^{\mathrm{l}}_{i},
\mu_{i}^{*,+}, \mu_{i}^{*,-} \bigr{\}}.
\end{equation*}

To conclude the proof, suppose now that $(h_{2}^{3})$ holds. 
Applying Lemma~\ref{lem-R}, the contradiction $u(\sigma_{i+1})>R$ follows when
\begin{equation*}
\mu > \bar{\mu}_{i} := \dfrac{\lambda \|a\|_{L^{1}(I^{+}_{i})} \Gamma(R) e^{2|c||I^{+}_{i}\cup I^{-}_{i}|}|I^{+}_{i}\cup I^{-}_{i}|}{\|A_{i}^{\mathrm{r}}\|_{L^{1}(I^{-}_{i})}\chi(r,R)}.
\end{equation*}
We conclude that the case $(h_{2}^{3})$ never occurs if
\begin{equation*}
\mu > \mu_{2}^{(H_{3})} : = \max_{i=1,\ldots,m} \bar{\mu}_{i}.
\end{equation*}

\smallskip

Summing up, we can apply Lemma~\ref{lem-deg1} for
\begin{equation}\label{def-H3}
\mu > \mu^{(H_{3})}:=\max \Bigl{\{} \mu_{1}^{(H_{3})},\mu_{2}^{(H_{3})}, \mu^{\#}(\lambda) \Bigr{\}}
\end{equation}
and therefore formula \eqref{eq-lem-deg1} is
verified.
\qed

\section{Globally defined solutions and symbolic dynamics}\label{section-5}

In this section we prove Theorem~\ref{th-chaos}. Actually, we are going to give just a sketch of the argument, which follows the same schemes of the one for the proof of \cite[Theorem~4.5]{FeZa-17jde}. 
We also remark that one could adapt to the present setting also the discussion developed in~\cite[Section~6]{BoFeZa-18tams}, in order to show that the existence of non-periodic bounded solutions coded by sequences of three symbols implies semiconjugation of a suitable map induced by $(\mathscr{E}_{\lambda,\mu})$ with the Bernoulli shift.

\begin{proof}[Proof of Theorem~\ref{th-chaos}]
Given $\rho>0$, we fix the constants $\lambda^{*}$, $r$, $R$, and $\mu^{*}$ as in Theorem~\ref{th-mult}. 
The first crucial observation is that all these constants depend (besides on $g$) only on the behavior of the weight function $a(x)$ on the intervals $I^{+}_{i}$ and $I^{-}_{i}$ with $i\in\{1,\ldots,m\}$ (and not on the length $P$ of the periodicity interval). As a consequence, the conclusion of Theorem~\ref{th-mult} holds (with the same constants) even if, in place of $\mathopen{[}0,P\mathclose{]}$, an interval of the type $\mathopen{[}n_{1} P,n_{2} P\mathclose{]}$ (with $n_{1},n_{2}\in\mathbb{Z}$ and $n_{1} < n_{2}$) is considered.

Let $\mathcal{S} = (\mathcal{S}_{i})_{i\in\mathbb{Z}}\in \{0,1,2\}^{\mathbb{Z}}$ be an arbitrary sequence which is not identically zero.

If $\mathcal{S}$ is $km$-periodic for some integer $k\geq1$, then an application of Theorem~\ref{th-mult} in the interval $\mathopen{[}0,kP\mathclose{]}$ ensures the existence of at least a $kP$-periodic solution $u_{\mathcal{S}}(x)$ of $(\mathscr{E}_{\lambda,\mu})$ coded by $\mathcal{S}$.

If it is not the case, we approximate $\mathcal{S}$ with the sequence $(\mathcal{S}^{n})_{n}$, where $\mathcal{S}^{n}\in \{0,1,2\}^{\mathbb{Z}}$ is the $(2n+1)m$-periodic sequence defined as
\begin{equation*}
\mathcal{S}^{n}_{j}:=\mathcal{S}_{j}, \quad \text{for $j=-nm+1,\ldots,(n+1)m$.}
\end{equation*}
An application of Theorem~\ref{th-mult} on the interval $\mathopen{[}-nP,(n+1)P\mathclose{]}$ (at least for $n$ sufficiently large, so that $\mathcal{S}^{n}\not\equiv0$) leads to the existence of a non-constant positive $(2n+1)P$-periodic solution $u_{n}(x)$ of $(\mathscr{E}_{\lambda,\mu})$ such that
\begin{itemize}
\item $\max_{t \in I^{+}_{i,\ell}} u_{n}(x) < r$, if $\mathcal{S}_{j} = 0$ for $j= i + \ell m$;
\item $r < \max_{t \in I^{+}_{i,\ell}} u_{n}(x) < \rho$, if $\mathcal{S}_{j} = 1$ for $j= i + \ell m$;
\item $\rho < \max_{t \in I^{+}_{i,\ell}} u_{n}(x) < R$, if $\mathcal{S}_{j} = 2$ for $j= i + \ell m$;
\end{itemize}
for every $i=1,\ldots,m$ and $\ell=-n,\ldots,n$.

A compactness argument (cf.~\cite[Section~4.3]{FeZa-17jde}) ensures the existence of a solution $\tilde{u}(x)$ of $(\mathscr{E}_{\lambda,\mu})$ defined on $\mathbb{R}$ and  obtained as the limit of a subsequence of $u_{n}(x)$.
Passing to the limit as $n\to\infty$, we have
\begin{itemize}
\item $\max_{x\in I^{+}_{i,\ell}} \tilde{u}(x) \leq r$, if $\mathcal{S}_{j} = 0$ for $j= i + \ell m$;
\item $r \leq \max_{x\in I^{+}_{i,\ell}} \tilde{u}(x) \leq \rho$, if $\mathcal{S}_{j} = 1$ for $j= i + \ell m$;
\item $\rho \leq \max_{x\in I^{+}_{i,\ell}} \tilde{u}(x) \leq R$, if $\mathcal{S}_{j} = 2$ for $j= i + \ell m$;
\end{itemize}
for every $i=1,\ldots,m$ and $\ell\in\mathbb{Z}$.

To conclude the proof we have to show that the above inequalities are strict. This can be done using on one hand Lemma~\ref{lem-rho} (ensuring that $\max_{I^{+}_{i,\ell}} u_{n} \neq \rho$) and on the other hand the arguments exploited in Section~\ref{section-4.1} to prove that the alternatives $(h^1_{1})$ and $(h^1_{3})$ can not hold (notice that for these the periodicity is not necessary).
\end{proof}

\begin{remark}\label{rem-5.1}
Given an integer $k\geq 2$, Theorem~\ref{th-mult} provides positive $kP$-periodic solutions of $(\mathscr{E}_{\lambda,\mu})$. In this direction, it is natural to investigate whether such solutions have $kP$ as minimal period, namely, whether they
are not $\ell P$-periodic for any integer $\ell = 1,\ldots,k-1$. A $kP$-periodic solution with this property is usually said to be a 
\textit{subharmonic solution of order $k$} (cf.~\cite{BoFe-18} and \cite[Section~4.1]{FeZa-17jde} for additional comments and references on the subject).

Given an integer $k\geq 2$, in order to produce at least a subharmonic solutions of order $k$, it is sufficient to take the $km$-periodic sequence $\mathcal{S} = (\mathcal{S}_{j})_{j\in\mathbb{Z}}\in\{0,1,2\}^{\mathbb{Z}}$ given by $\mathcal{S}_{1}=1$ and $\mathcal{S}_{j}=0$ for $j\in\{2,\ldots,km\}$. The minimality of the period $kP$ is a consequence of the behavior of the solution $u_{\mathcal{S}}(x)$ given by $\mathcal{S}$. Following the discussion developed in \cite[Section~6]{BoFeZa-18tams} and in \cite[Section~4.2]{FeZa-17jde}, one can give an estimate for the number of subharmonic solutions of order $k$. Indeed, one can define a one-to-one correspondence between the aperiodic necklaces of length $k$ on $n$ colors and the non-null strings of length $k$ on $n$ symbols. Taking $n=3^{m}$ symbols/colors, the desired estimate is given by Witt's formula:
\begin{equation*}
\Sigma_{3^{m}}(k) = \dfrac{1}{k} \sum_{l|k} \mu(l) \, 3^{\frac{mk}{l}},
\end{equation*}
where $\mu(\cdot)$ is the M\"{o}bius function, defined on $\mathbb{N}\setminus\{0\}$ by $\mu(1) = 1$,
$\mu(l) = (-1)^{q}$ if $l$ is the
product of $q$ distinct primes and $\mu(l) = 0$ otherwise. We refer to \cite[Remark~4.1]{Fe-18dcdss} for an interesting discussion on this formula.
$\hfill\lhd$
\end{remark}

\section{Related results and remarks}\label{section-6}

We conclude the paper with some complementary results and remarks.

\subsection{Subharmonic solutions}\label{section-6.1}

In the context of Theorem~\ref{th-exis}, if we further suppose that $g(u)$ is of class $\mathcal{C}^{2}$ in an interval $\mathopen{[}0,\varepsilon\mathclose{]}$ and satisfies $g''(u) > 0$ for every $u \in \mathopen{]}0,\varepsilon\mathclose{]}$, then the equation 
\begin{equation}\label{eq-c0}
u'' + a_{\lambda,\mu}(x) g(u) = 0
\end{equation}
has, for every $\lambda > \lambda^*$ and $\mu > \mu^{\#}(\lambda)$, positive subharmonic solutions of order $k$ for any integer $k$ large enough.

This follows from \cite[Theorem~3.3]{BoFe-18}, after having observed that the constant $\lambda^*$ given therein does not depend on $a^{-}(x)$ (actually, is obtained exactly as in Lemma~\ref{lem-rho}). Let us stress that such a proof is of symplectic nature, being based on the Poincar\'e--Birkhoff fixed point theorem: therefore, the assumption $c = 0$ is essential. Subharmonic solutions in the case $c \neq 0$ can be found as in Remark~\ref{rem-5.1} (for every integer $k \geq 1$), but only for larger $\mu$, i.e.,~$\mu >\mu^{*}(\lambda)$.

\subsection{Dirichlet and Neumann boundary conditions}\label{section-6.2}

A suitable variant of Theorem~\ref{th-mult} is valid when equation
$(\mathscr{E}_{\lambda,\mu})$ is coupled with Dirichlet boundary conditions
\begin{equation*}
u(0) = u(P) = 0
\end{equation*}
or Neumann boundary conditions
\begin{equation*}
u'(0) = u'(P) = 0.
\end{equation*}
Let us recall that, in both these cases, with a standard change of variable we can assume $c = 0$ (cf.~\cite[Appendix C]{Fe-18}).

In this context, it is possible to consider a slightly more general sign condition, with respect to $(a_*)$, for the $L^{1}$-weight
$a \colon \mathopen{[}0,P\mathclose{]} \to \mathbb{R}$. Precisely, $a(x)$
can be allowed to have an initial negativity interval $I^{-}_{0}$ and, if $m \geq 2$ or $I^{-}_{0} \neq \emptyset$, to have $I^{-}_{m}=\emptyset$, that is, $a(x)$ can be non-negative in a left neighborhood of $P$, provided that there exists at least one negativity interval (cf.~\cite[Section~7.2]{BoFeZa-18tams}).

The proofs require just minor modifications with respect to the ones given for the periodic problem. 
Precisely, the appropriate abstract setting for Dirichlet and Neumann boundary conditions is described in \cite[Remark~2.1]{BoFeZa-18tams}; with this in mind, the general strategy in Section~\ref{section-4} remains the same.
In order to verify the assumptions of the degree lemmas in Section~\ref{section-2.1}, the estimates given in Section~\ref{section-2.2} can still be exploited, since they are of local nature, and the boundary condition at $x = 0$ and $x = P$ can be used in place of the $P$-periodicity to reach the desired contradictions. See also Figure~\ref{fig-1} for a numerical example.

As standard corollaries, one can give multiplicity results for radially symmetric positive solutions of elliptic BVPs on annular domains (cf.~\cite[Section~7.3]{BoFeZa-18tams} and \cite[Section~3]{FeSo-18na}).

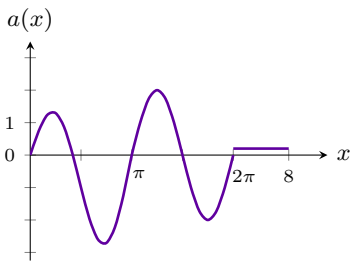
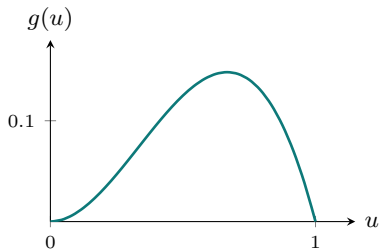
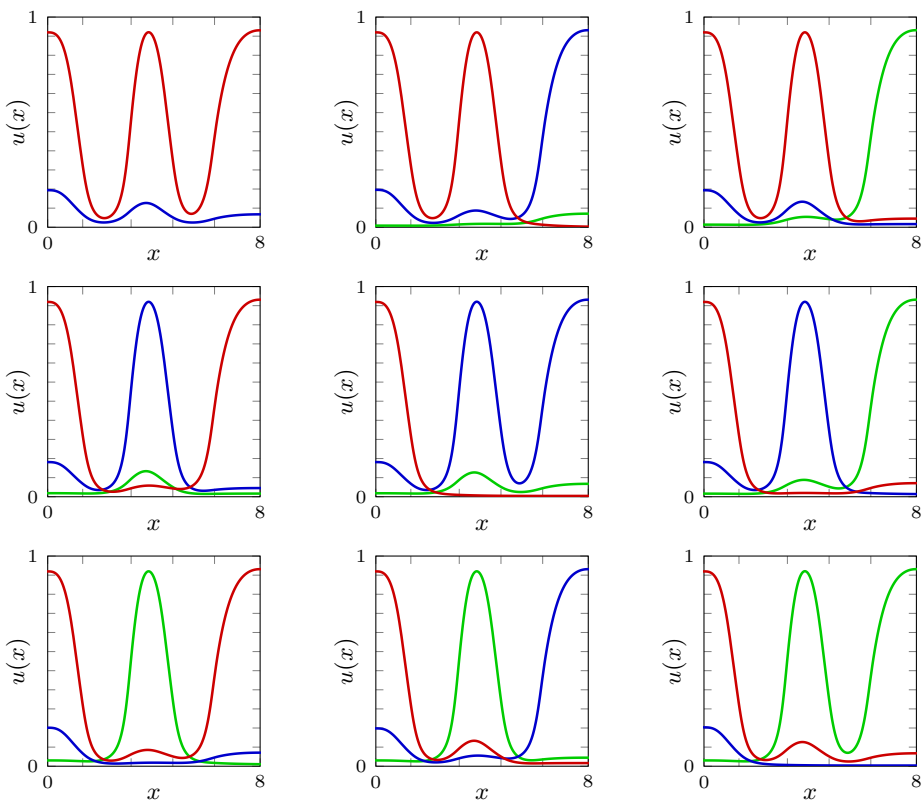
\begin{figure}[!htb]
\centering
\begin{subfigure}[t]{0.44\textwidth}
\centering
\begin{tikzpicture}[scale=1]
\begin{axis}[
  tick label style={font=\scriptsize},
  axis y line=left, 
  axis x line=middle,
  xtick={1.5708, 3.14159, 4.71239, 6.28319,8},
  ytick={-3,-2,-1,0,1,2,3},
  xticklabels={,$\,\,\,\,\pi$, , $\,\,\,\,\,\,2\pi$,$8$},
  yticklabels={,,,$0$,$1$,},
  xlabel={\small $x$},
  ylabel={\small $a(x)$},
every axis x label/.style={
    at={(ticklabel* cs:1.0)},
    anchor=west,
},
every axis y label/.style={
    at={(ticklabel* cs:1.0)},
    anchor=south,
},
  width=5.5cm,
  height=4.5cm,
  xmin=0,
  xmax=9.2,
  ymin=-3.25,
  ymax=3.5]
\addplot [mark=none,domain=0:6.283,line width=1pt,smooth,color={rgb:red,0.3;green,0;blue,0.5}] {2*sin(deg(2*x))-max(0,sin(deg(x)))};
\addplot [mark=none,domain=6.283:8,line width=1pt,smooth,color={rgb:red,0.3;green,0;blue,0.5}] {0.2};
\end{axis}
\end{tikzpicture}
\caption{Graph of the weight term defined as $a(x)=2\sin(2x)-\max\{0,\sin(x)\}$ on $\mathopen{[}0,2\pi\mathclose{]}$ and $a(x)=0.2$ on $\mathopen{[}2\pi,8\mathclose{]}$.}
\end{subfigure}
\hspace*{\fill}
\begin{subfigure}[t]{0.44\textwidth}
\centering
\begin{tikzpicture}[scale=1]
\begin{axis}[
  tick label style={font=\scriptsize},
  axis y line=left, 
  axis x line=bottom,
  xtick={0,1},
  ytick={0.1},
  xlabel={\small $u$},
  ylabel={\small $g(u)$},
every axis x label/.style={
    at={(ticklabel* cs:1.0)},
    anchor=west,
},
every axis y label/.style={
    at={(ticklabel* cs:1.0)},
    anchor=south,
},
  width=5.6cm, 
  height=4cm,
  xmin=0,
  xmax=1.15,
  ymin=0,
  ymax=0.18] 
\addplot [mark=none,domain=0:1,line width=1pt,smooth,color={rgb:red,0.1;green,0.9;blue,0.9}] {x*x*(1-x)};
\end{axis}
\end{tikzpicture}
\caption{Graph of the nonlinear term $g(u)=u^{2}(1-u)$.} 
\end{subfigure}
\\
\vspace{15pt}
\begin{subfigure}[t]{1\textwidth}
\centering
\begin{tikzpicture}[scale=1]
\begin{axis}[
  tick label style={font=\scriptsize,major tick length=3pt},
          scale only axis,
  enlargelimits=false,
  xtick={0, 1.31812,3.14159,4.71239,6.28319,8},
  xticklabels={0,,,,,$8$},
  ytick={0,1},
  max space between ticks=50,
                minor x tick num=3,
                minor y tick num=10,                
  xlabel={\small $x$},
  ylabel={\small $u(x)$},
every axis x label/.style={
below,
at={(1.4cm,0.1cm)},
  yshift=-8pt
  },
every axis y label/.style={
below,
at={(0cm,1.4cm)},
  xshift=-3pt},
  y label style={rotate=90,anchor=south},
  width=2.8cm,
  height=2.8cm,  
  xmin=0,
  xmax=8,
  ymin=0,
  ymax=1]
\addplot graphics[xmin=0,xmax=8,ymin=0,ymax=1] {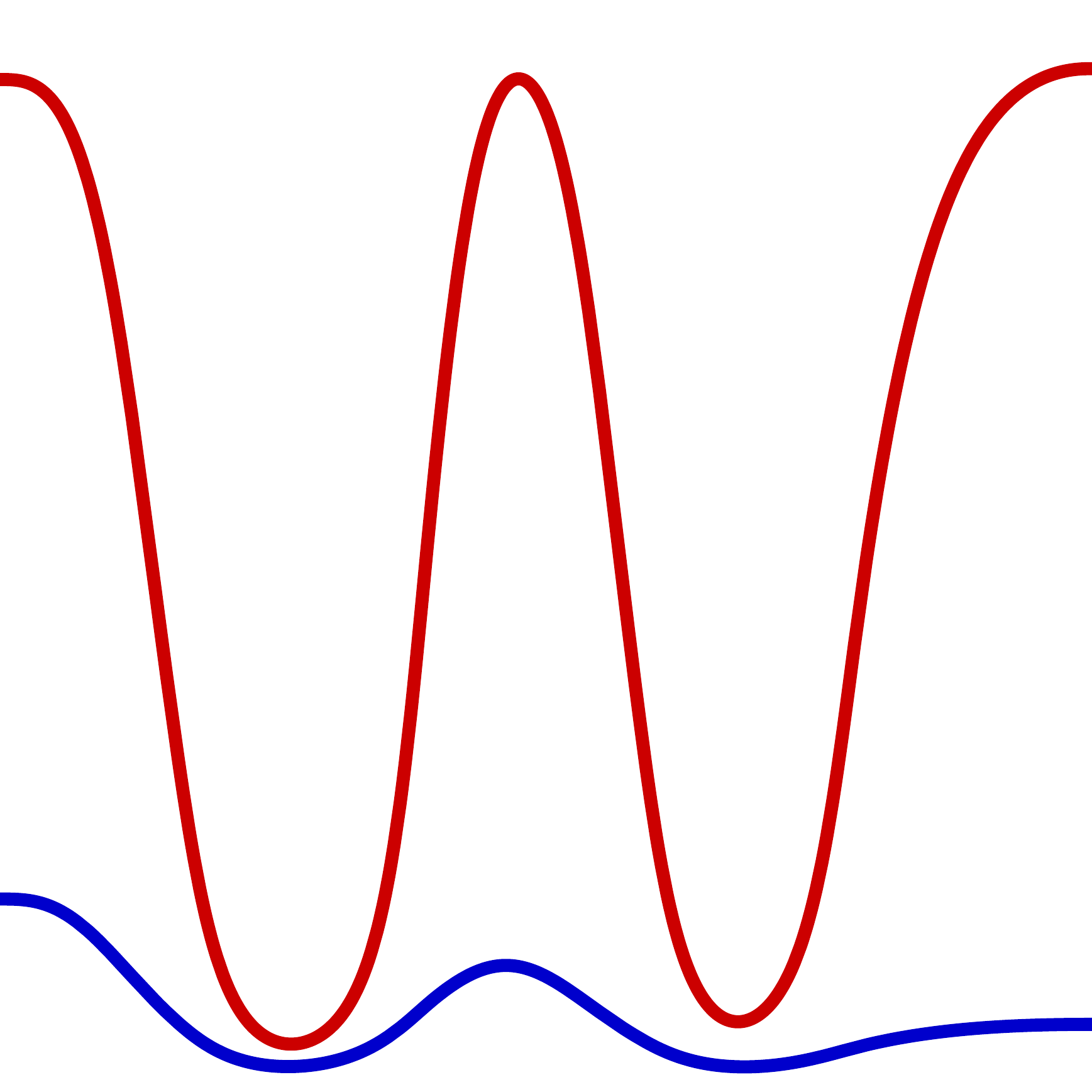};
\end{axis}
\end{tikzpicture} 
\hspace*{\fill}
\begin{tikzpicture}[scale=1]
\begin{axis}[
  tick label style={font=\scriptsize,major tick length=3pt},
          scale only axis,
  enlargelimits=false,
  xtick={0, 1.31812,3.14159,4.71239,6.28319,8},
  xticklabels={0,,,,,$8$},
  ytick={0,1},
  max space between ticks=50,
                minor x tick num=3,
                minor y tick num=10,                
  xlabel={\small $x$},
  ylabel={\small $u(x)$},
every axis x label/.style={
below,
at={(1.4cm,0.1cm)},
  yshift=-8pt
  },
every axis y label/.style={
below,
at={(0cm,1.4cm)},
  xshift=-3pt},
  y label style={rotate=90,anchor=south},
  width=2.8cm,
  height=2.8cm,  
  xmin=0,
  xmax=8,
  ymin=0,
  ymax=1] 
\addplot graphics[xmin=0,xmax=8,ymin=0,ymax=1] {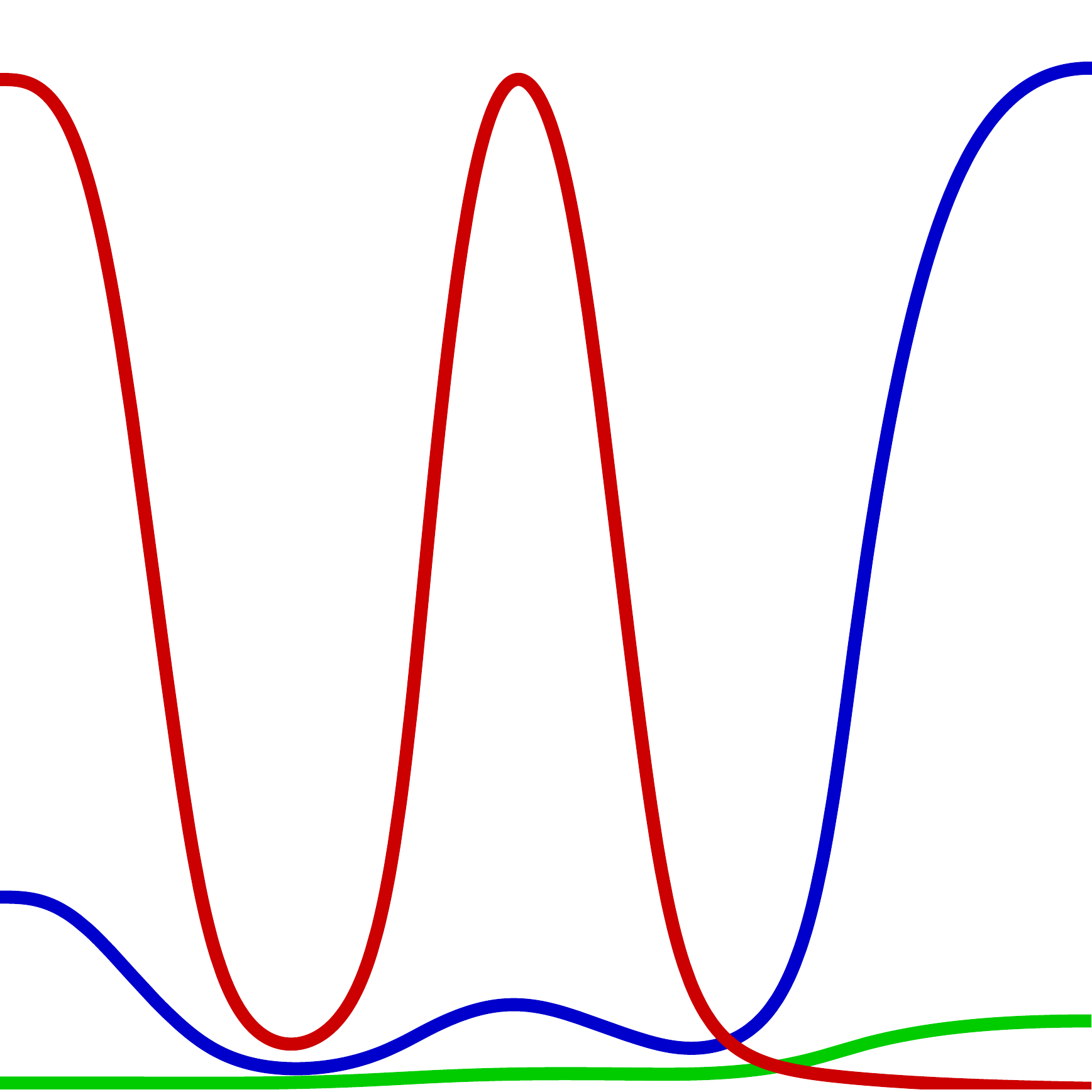};
\end{axis}
\end{tikzpicture} 
\hspace*{\fill}
\begin{tikzpicture}[scale=1]
\begin{axis}[
  tick label style={font=\scriptsize,major tick length=3pt},
          scale only axis,
  enlargelimits=false,
  xtick={0, 1.31812,3.14159,4.71239,6.28319,8},
  xticklabels={0,,,,,$8$},
  ytick={0,1},
  max space between ticks=50,
                minor x tick num=3,
                minor y tick num=10,                
  xlabel={\small $x$},
  ylabel={\small $u(x)$},
every axis x label/.style={
below,
at={(1.4cm,0.1cm)},
  yshift=-8pt
  },
every axis y label/.style={
below,
at={(0cm,1.4cm)},
  xshift=-3pt},
  y label style={rotate=90,anchor=south},
  width=2.8cm,
  height=2.8cm,  
  xmin=0,
  xmax=8,
  ymin=0,
  ymax=1] 
\addplot graphics[xmin=0,xmax=8,ymin=0,ymax=1] {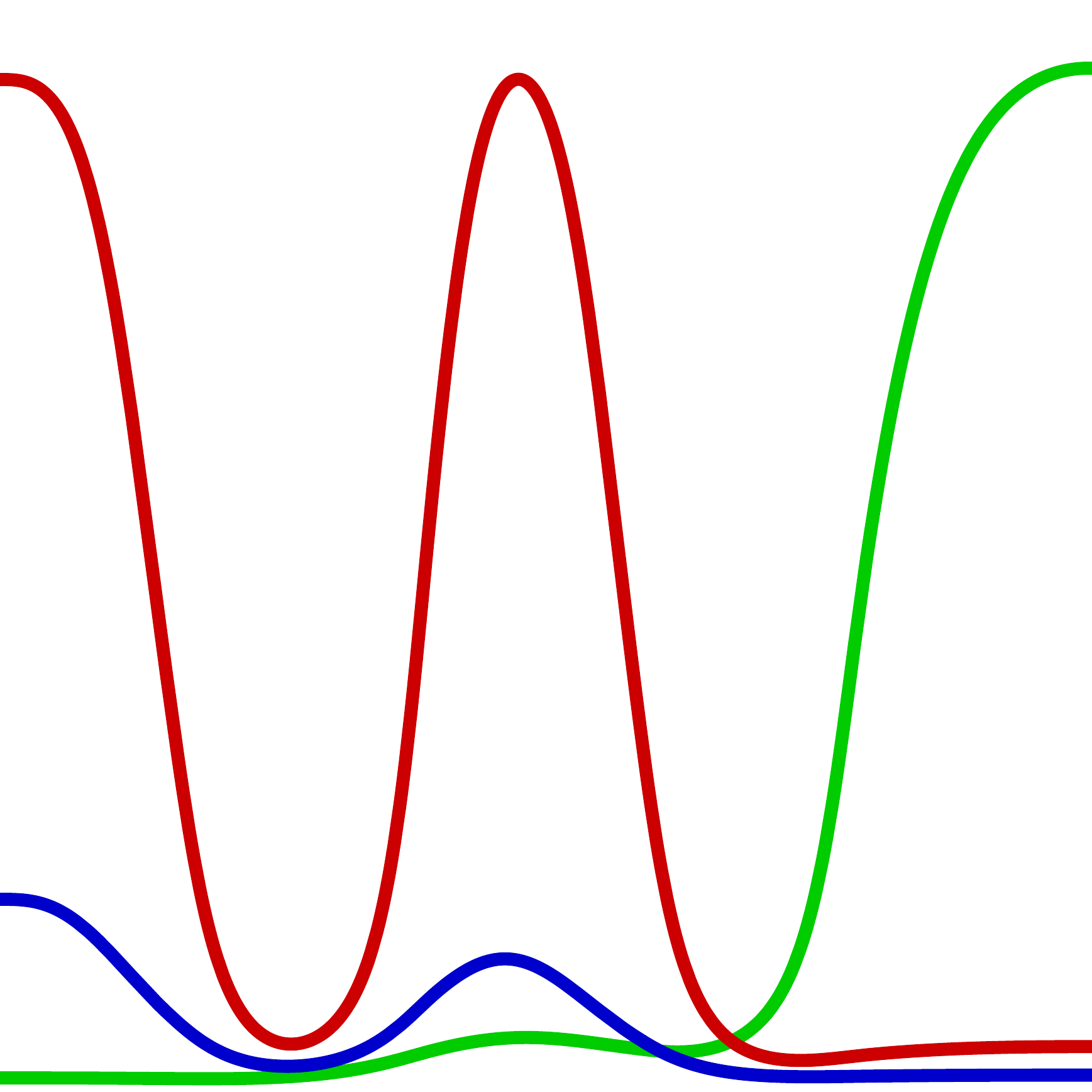};
\end{axis}
\end{tikzpicture} 
\\
\begin{tikzpicture}[scale=1]
\begin{axis}[
  tick label style={font=\scriptsize,major tick length=3pt},
          scale only axis,
  enlargelimits=false,
  xtick={0, 1.31812,3.14159,4.71239,6.28319,8},
  xticklabels={0,,,,,$8$},
  ytick={0,1},
  max space between ticks=50,
                minor x tick num=3,
                minor y tick num=10,                
  xlabel={\small $x$},
  ylabel={\small $u(x)$},
every axis x label/.style={
below,
at={(1.4cm,0.1cm)},
  yshift=-8pt
  },
every axis y label/.style={
below,
at={(0cm,1.4cm)},
  xshift=-3pt},
  y label style={rotate=90,anchor=south},
  width=2.8cm,
  height=2.8cm,  
  xmin=0,
  xmax=8,
  ymin=0,
  ymax=1] 
\addplot graphics[xmin=0,xmax=8,ymin=0,ymax=1] {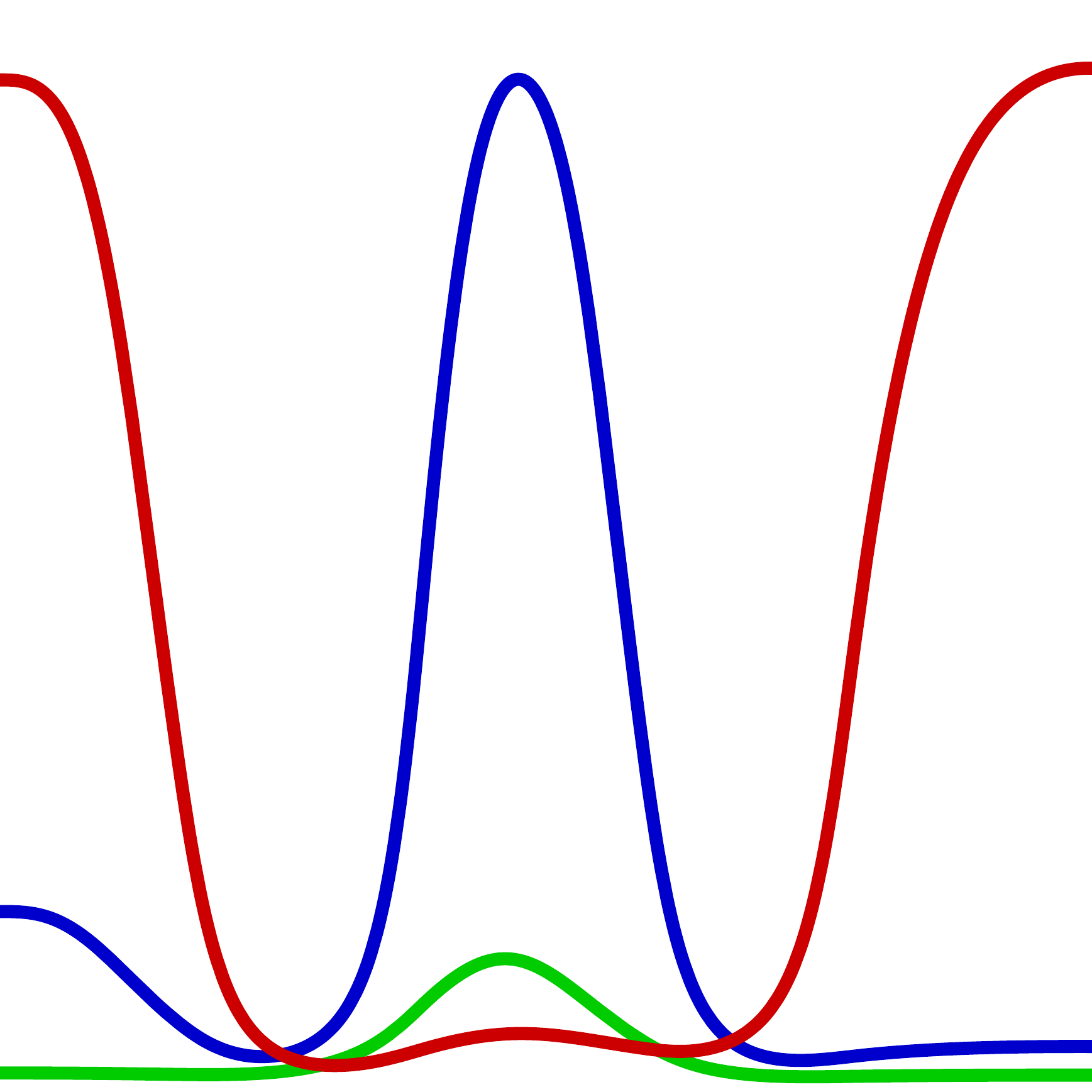};
\end{axis}
\end{tikzpicture} 
\hspace*{\fill}
\begin{tikzpicture}[scale=1]
\begin{axis}[
  tick label style={font=\scriptsize,major tick length=3pt},
          scale only axis,
  enlargelimits=false,
  xtick={0, 1.31812,3.14159,4.71239,6.28319,8},
  xticklabels={0,,,,,$8$},
  ytick={0,1},
  max space between ticks=50,
                minor x tick num=3,
                minor y tick num=10,                
  xlabel={\small $x$},
  ylabel={\small $u(x)$},
every axis x label/.style={
below,
at={(1.4cm,0.1cm)},
  yshift=-8pt
  },
every axis y label/.style={
below,
at={(0cm,1.4cm)},
  xshift=-3pt},
  y label style={rotate=90,anchor=south},
  width=2.8cm,
  height=2.8cm,  
  xmin=0,
  xmax=8,
  ymin=0,
  ymax=1] 
\addplot graphics[xmin=0,xmax=8,ymin=0,ymax=1] {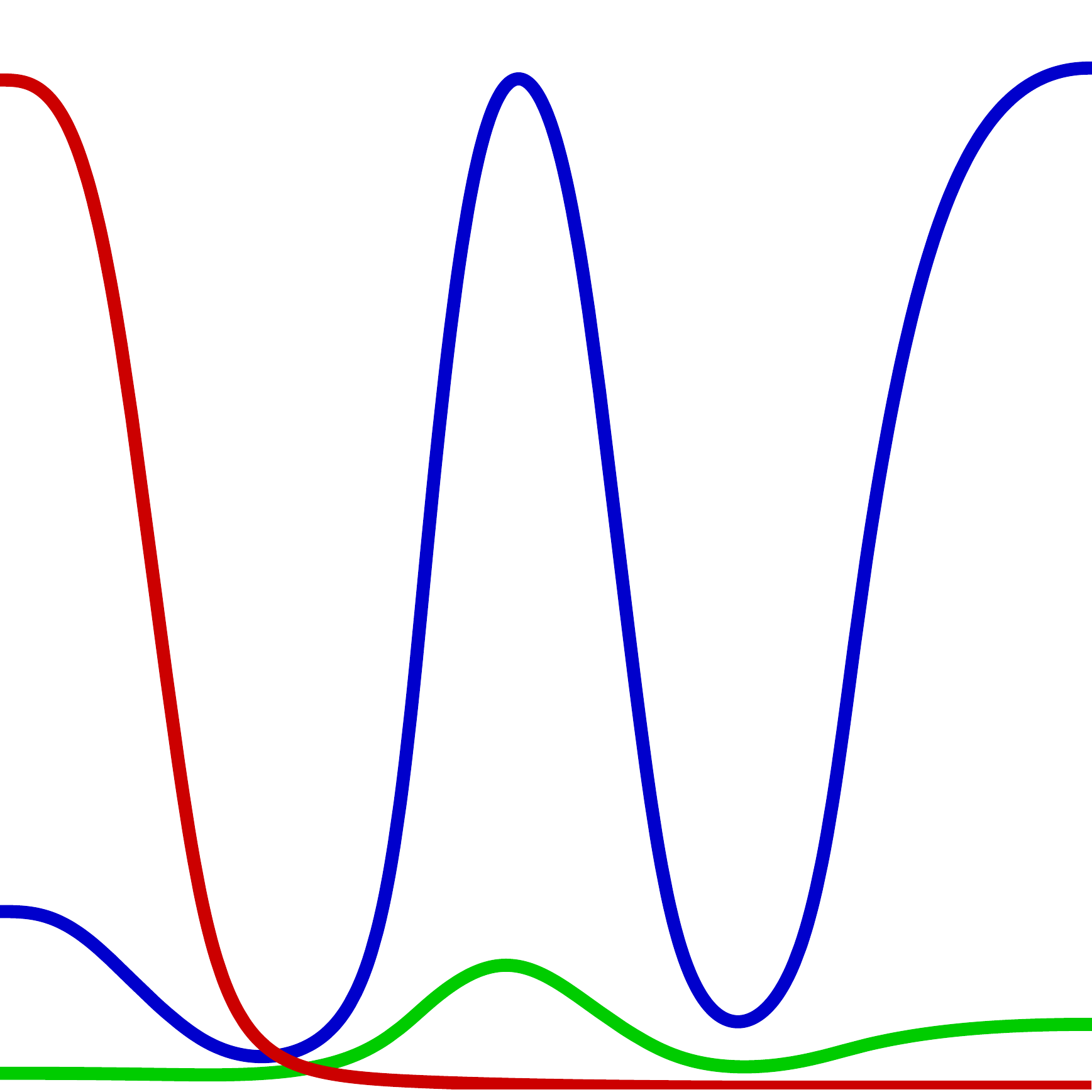};
\end{axis}
\end{tikzpicture} 
\hspace*{\fill}
\begin{tikzpicture}[scale=1]
\begin{axis}[
  tick label style={font=\scriptsize,major tick length=3pt},
          scale only axis,
  enlargelimits=false,
  xtick={0, 1.31812,3.14159,4.71239,6.28319,8},
  xticklabels={0,,,,,$8$},
  ytick={0,1},
  max space between ticks=50,
                minor x tick num=3,
                minor y tick num=10,                
  xlabel={\small $x$},
  ylabel={\small $u(x)$},
every axis x label/.style={
below,
at={(1.4cm,0.1cm)},
  yshift=-8pt
  },
every axis y label/.style={
below,
at={(0cm,1.4cm)},
  xshift=-3pt},
  y label style={rotate=90,anchor=south},
  width=2.8cm,
  height=2.8cm,  
  xmin=0,
  xmax=8,
  ymin=0,
  ymax=1] 
\addplot graphics[xmin=0,xmax=8,ymin=0,ymax=1] {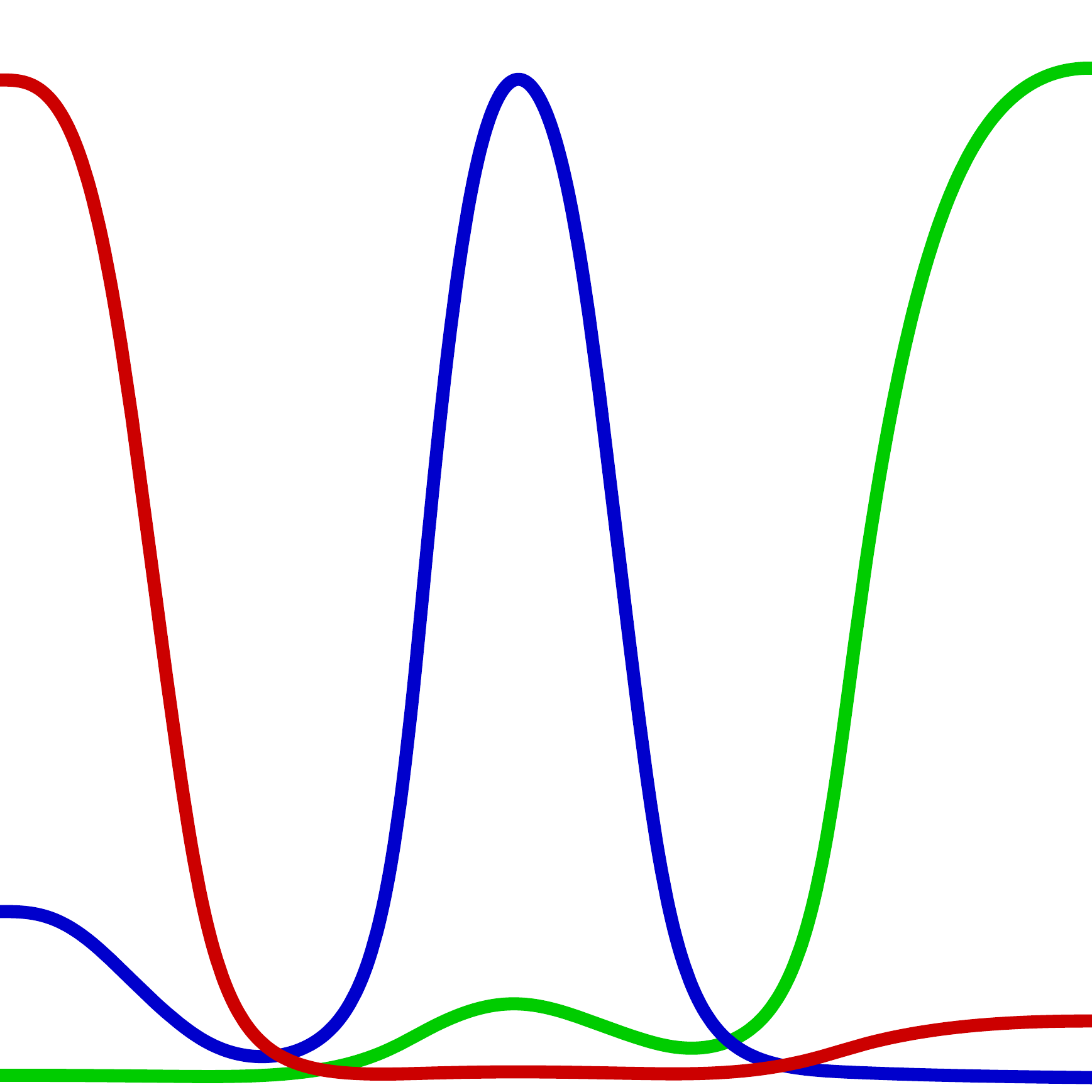};
\end{axis}
\end{tikzpicture} 
\\
\begin{tikzpicture}[scale=1]
\begin{axis}[
  tick label style={font=\scriptsize,major tick length=3pt},
          scale only axis,
  enlargelimits=false,
  xtick={0, 1.31812,3.14159,4.71239,6.28319,8},
  xticklabels={0,,,,,$8$},
  ytick={0,1},
  max space between ticks=50,
                minor x tick num=3,
                minor y tick num=10,                
  xlabel={\small $x$},
  ylabel={\small $u(x)$},
every axis x label/.style={
below,
at={(1.4cm,0.1cm)},
  yshift=-8pt
  },
every axis y label/.style={
below,
at={(0cm,1.4cm)},
  xshift=-3pt},
  y label style={rotate=90,anchor=south},
  width=2.8cm,
  height=2.8cm,  
  xmin=0,
  xmax=8,
  ymin=0,
  ymax=1] 
\addplot graphics[xmin=0,xmax=8,ymin=0,ymax=1] {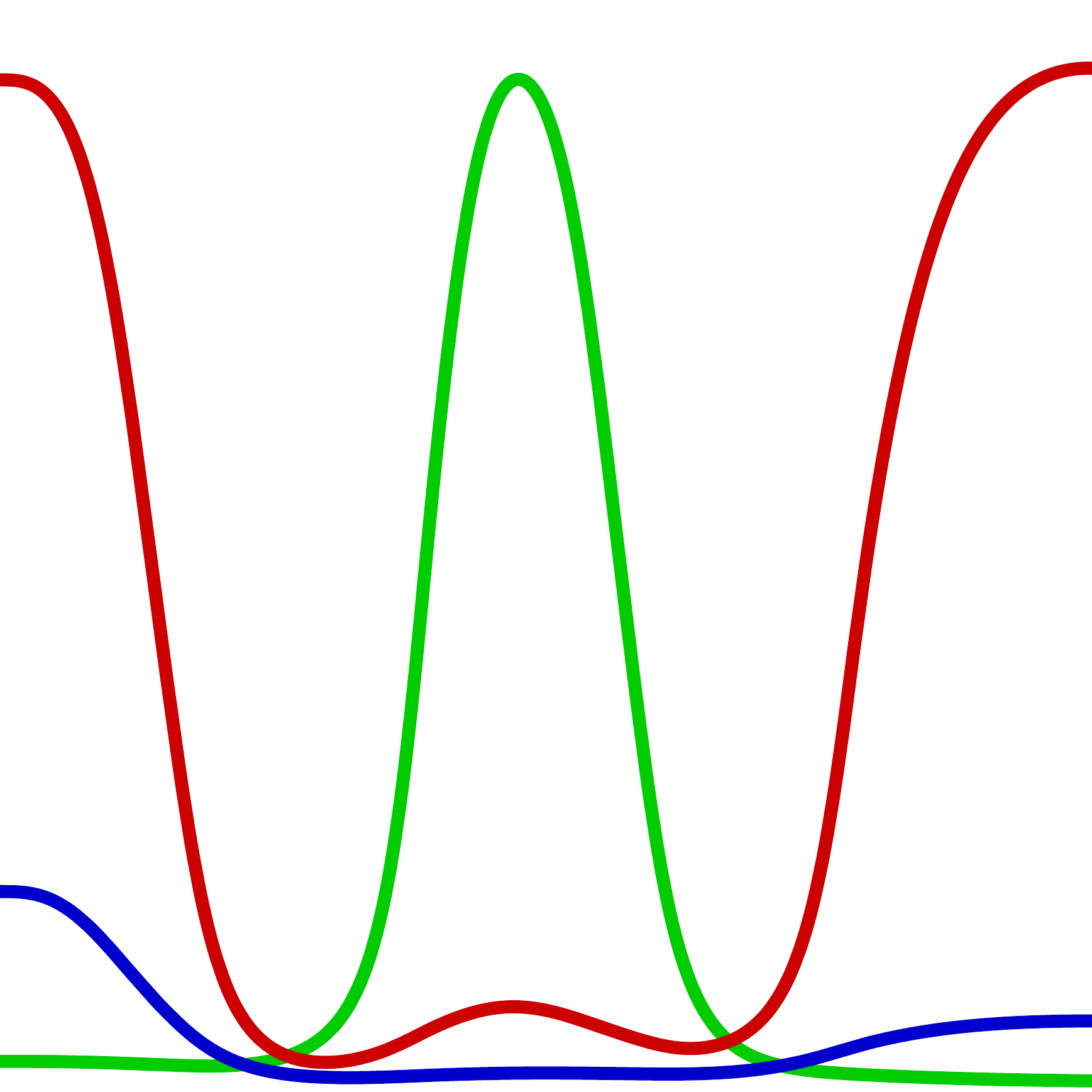};
\end{axis}
\end{tikzpicture} 
\hspace*{\fill}
\begin{tikzpicture}[scale=1]
\begin{axis}[
  tick label style={font=\scriptsize,major tick length=3pt},
          scale only axis,
  enlargelimits=false,
  xtick={0, 1.31812,3.14159,4.71239,6.28319,8},
  xticklabels={0,,,,,$8$},
  ytick={0,1},
  max space between ticks=50,
                minor x tick num=3,
                minor y tick num=10,                
  xlabel={\small $x$},
  ylabel={\small $u(x)$},
every axis x label/.style={
below,
at={(1.4cm,0.1cm)},
  yshift=-8pt
  },
every axis y label/.style={
below,
at={(0cm,1.4cm)},
  xshift=-3pt},
  y label style={rotate=90,anchor=south},
  width=2.8cm,
  height=2.8cm,  
  xmin=0,
  xmax=8,
  ymin=0,
  ymax=1] 
\addplot graphics[xmin=0,xmax=8,ymin=0,ymax=1] {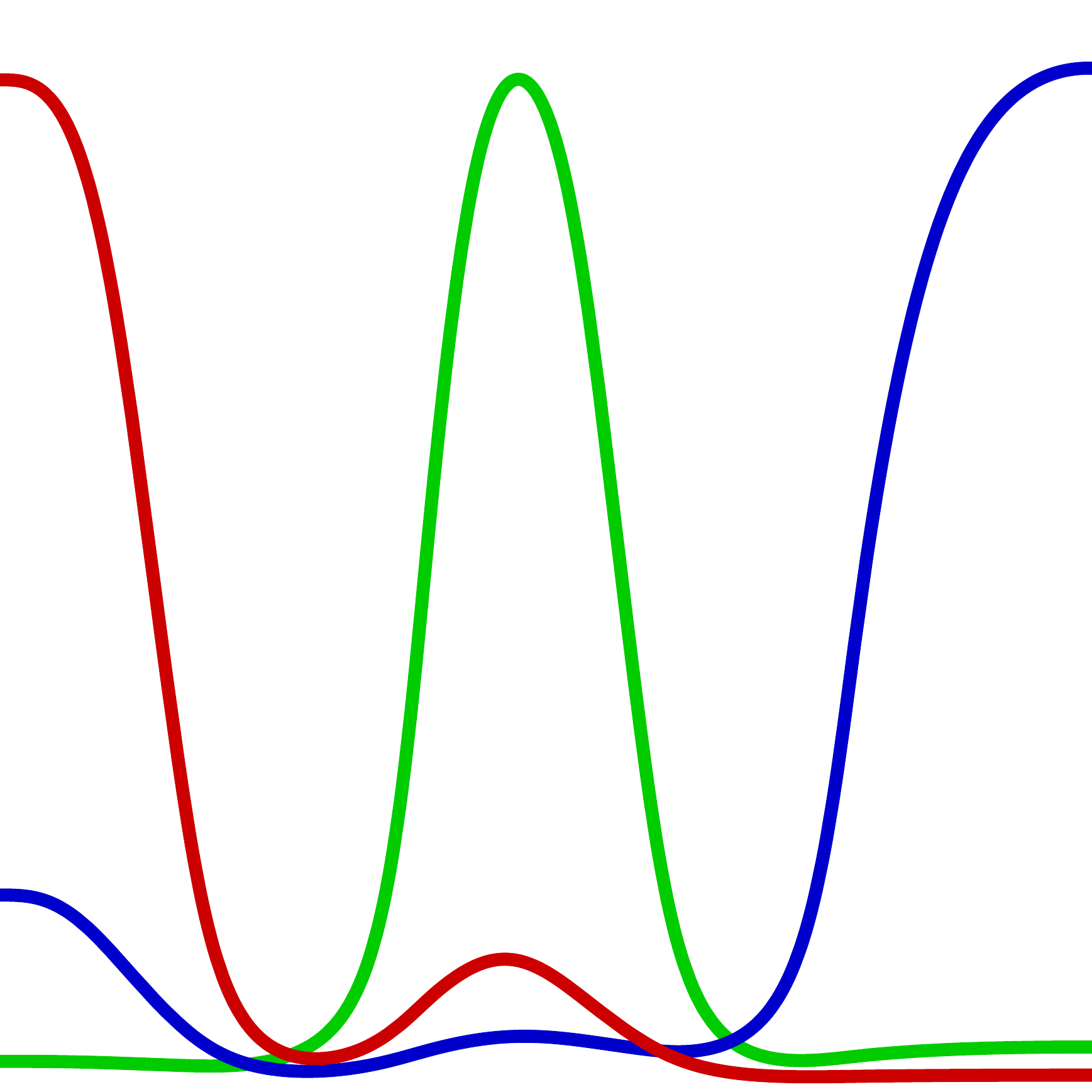};
\end{axis}
\end{tikzpicture} 
\hspace*{\fill}
\begin{tikzpicture}[scale=1]
\begin{axis}[
  tick label style={font=\scriptsize,major tick length=3pt},
          scale only axis,
  enlargelimits=false,
  xtick={0, 1.31812,3.14159,4.71239,6.28319,8},
  xticklabels={0,,,,,$8$},
  ytick={0,1},
  max space between ticks=50,
                minor x tick num=3,
                minor y tick num=10,                
  xlabel={\small $x$},
  ylabel={\small $u(x)$},
every axis x label/.style={
below,
at={(1.4cm,0.1cm)},
  yshift=-8pt
  },
every axis y label/.style={
below,
at={(0cm,1.4cm)},
  xshift=-3pt},
  y label style={rotate=90,anchor=south},
  width=2.8cm,
  height=2.8cm,  
  xmin=0,
  xmax=8,
  ymin=0,
  ymax=1] 
\addplot graphics[xmin=0,xmax=8,ymin=0,ymax=1] {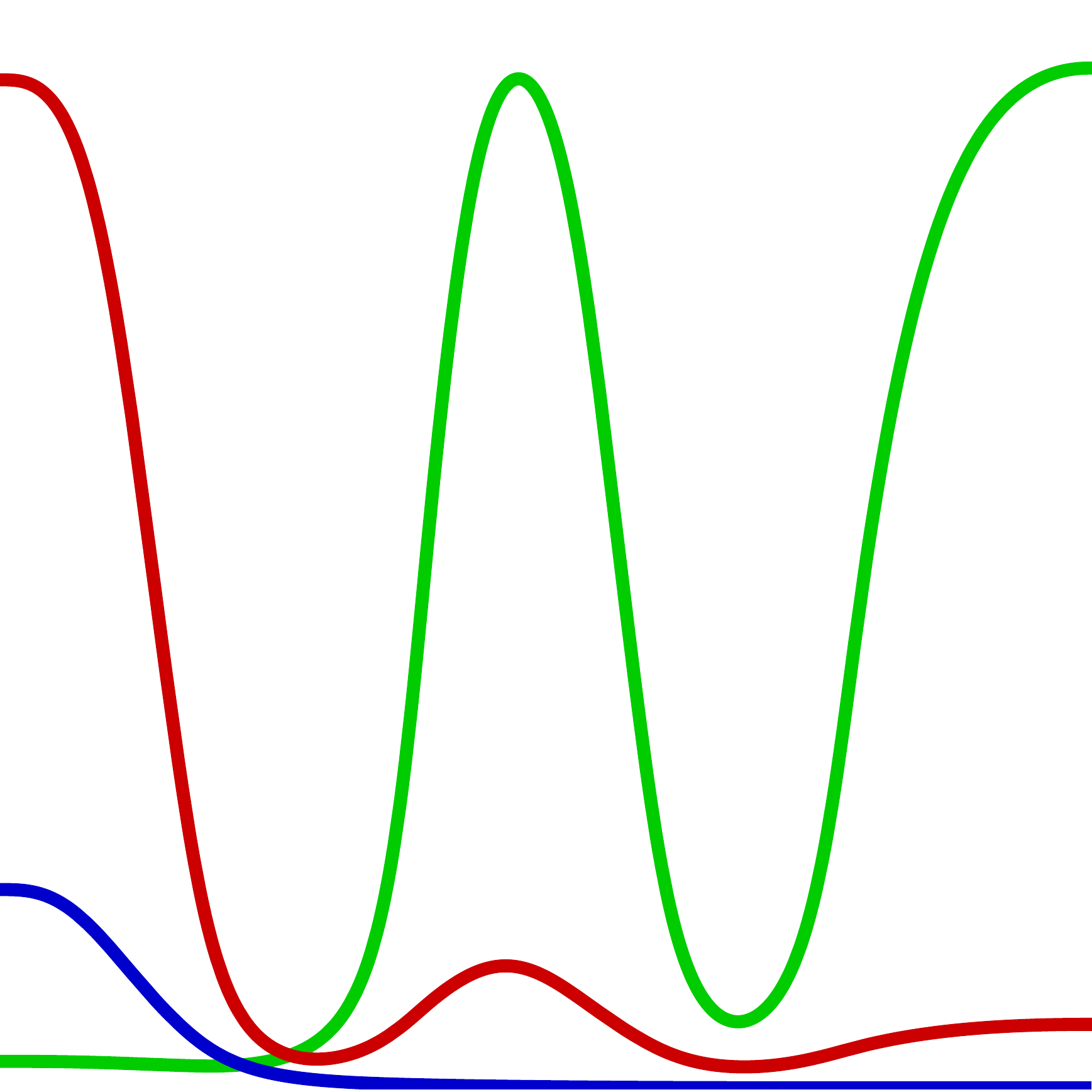};
\end{axis}
\end{tikzpicture} 
\caption{Graphs of $26 = 3^3-1$ positive solutions of the Neumann boundary value problem associated with $(\mathscr{E}_{\lambda,\mu})$, where $c = 1$, $a(x)$ is as in sub-figure (A) with $P=8$ (and so $m=3$), $g(u)$ is as in sub-figure (B), $\lambda=12$, and $\mu=80$.}
\end{subfigure} 
\caption{High multiplicity of positive solutions for the indefinite Neumann boundary value problem associated with $(\mathscr{E}_{\lambda,\mu})$. 
}     
\label{fig-1}
\end{figure}

\subsection{Stability issues}\label{section-6.3}

Dealing with equation \eqref{eq-c0} and assuming further that $g(u)$ is of class $\mathcal{C}^{2}$ in an interval $\mathopen{[}0,\varepsilon\mathclose{]}$ and satisfies
$g''(u) > 0$ for every $u \in \mathopen{]}0,\varepsilon\mathclose{]}$, some information about the linear (in)stability
of the solutions found in Theorem~\ref{th-exis} and Theorem~\ref{th-mult} can be given.
Here, linear stability/instability is meant in the sense of steady states of the corresponding parabolic problem, that is, a $P$-periodic solution $u(x)$ of \eqref{eq-c0}
is said to be \textit{linearly stable} (respectively, \textit{linearly unstable}) if the principal eigenvalue $\nu_{0}$ of the $P$-periodic problem associated with
\begin{equation*}
v'' + \bigl{(} \nu + a_{\lambda,\mu}(x) g'(u(x)) \bigr{)} v = 0
\end{equation*}
satisfies $\nu_{0} \geq 0$ (respectively, $\nu_{0} < 0$), cf.~\cite[Definition 2.1]{LGMMTe-13}. 
The same definition can be given when \eqref{eq-c0} is considered together with Dirichlet or Neumann boundary conditions (of course, the principal eigenvalue is meant with respect to the corresponding boundary conditions). It is worth mentioning that, for $P$-periodic solutions, this notion of linear stability is completely unrelated with respect to the more traditional one, based on Floquet theory,
arising as the linear version of Lyapunov stability \cite{Or-17}.

Taking into account the above discussion, one can apply \cite[Lemma~4.2]{BoFe-18} ensuring that $\nu_{0} < 0$ for every positive $P$-periodic solution $u(x)$ of \eqref{eq-c0} satisfying $\| u \|_\infty < \varepsilon$. Therefore, choosing $\rho \in \mathopen{]}0,\varepsilon\mathclose{[}$ in Theorem~\ref{th-mult}, we conclude that all the $2^m-1$ solutions 
associated with the strings $\mathcal{S}$ with $\mathcal{S}_{i} \neq 2$ for all $i=1,\ldots,m$, are linearly unstable (recall that, by property \eqref{eq-flower}, these solutions satisfy $\| u \|_\infty < \rho$). By a careful checking of the computation in \cite[Lemma~4.2]{BoFe-18}, one can deduce the same conclusion when Dirichlet/Neumann boundary conditions are taken into account.

In the same way we can also deduce that the small solution $u_{s}(x)$ in Theorem~\ref{th-exis} is linearly unstable: this is consistent with~\cite[Theorem 1.3]{LNS-10}, proving, for the Neumann problem, that one solution is unstable (while a second one is stable). 

\subsection{Asymptotic analysis}\label{section-6.4}

Using the arguments described in \cite[Section~5]{BoFeZa-18tams} and in \cite[Section~3.5]{FeZa-17jde}, it is possible to investigate the asymptotic behavior for $\mu \to +\infty$ of the solutions provided by Theorem~\ref{th-mult} and Theorem~\ref{th-chaos} (with $\lambda > \lambda^*$ fixed). More precisely, if $\{ u_{\mathcal{S},\mu}(x) \}_{\mu > \mu^{*}(\lambda)}$ denotes a family of solutions coded by the same string 
$\mathcal{S}$, one can show that, up to subsequences, the following hold:
\begin{itemize}
\item $u_{\mathcal{S},\mu}(x)$ converges to zero uniformly in all the negativity intervals of $a(x)$;
\item $u_{\mathcal{S},\mu}(x)$ converges to zero uniformly in the positivity intervals $I^{+}_{i,\ell}$ such that $\mathcal{S}_{i+\ell m} = 0$;
\item $u_{\mathcal{S},\mu}(x)$ converges to a positive solution of the Dirichlet problem associated with $(\mathscr{E}_{\lambda,\mu})$ 
on the positivity intervals $I^{+}_{i,\ell}$ such that $\mathcal{S}_{i+\ell m} \in \{1,2\}$ (notice that, from this discussion, it follows that such Dirichlet problems have at least two positive solutions, cf.~\cite{Ra-7374}).
\end{itemize}
Similarly, one can discuss the case of Dirichlet and Neumann boundary conditions (in the Neumann case, whenever $a(x)$ starts or ends with a positivity interval $I^{+}_{i}$ with corresponding $\mathcal{S}_{i} \in \{1,2\}$, then $u_{\mathcal{S},\mu}(x)$ converges in such an interval to a positive solution of a mixed Dirichlet/Neumann problem). We omit the details for briefness.

It is worth mentioning that, for the one-parameter equation
\begin{equation*}
u'' + \lambda a(x) g(u) = 0,
\end{equation*}
that is, equation $(\mathscr{E}_{\lambda,\mu})$ for $\lambda = \mu$ and $c = 0$, the asymptotic behavior of the 
two positive solutions when $\lambda \to +\infty$ has been carefully investigated in \cite{NNS-10}. Roughly speaking, the small solution
$u_{s}(x)$ converges to zero uniformly in the whole $\mathopen{[}0,P\mathclose{]}$, while the large solution $u_{\ell}(x)$ converges to $1$ (respectively, to $0$) 
uniformly on every compact subinterval of the interior of the positivity intervals (respectively, negativity intervals), see \cite[Theorem~1.3]{NNS-10} for the precise statement. Of course, this result is unrelated with the one discussed above for the two-parameter equation $(\mathscr{E}_{\lambda,\mu})$, since in the latter case $\lambda$ is fixed (and $\mu \to +\infty$). See also Figure~\ref{fig-2} for a numerical investigation.

\begin{figure}[!htb]
\centering
\begin{subfigure}[t]{0.45\textwidth}
\centering
\begin{tikzpicture}[scale=1]
\begin{axis}[
  tick label style={font=\scriptsize},
          scale only axis,
  enlargelimits=false,
  xtick={0, 1.31812,3.14159},
  xticklabels={0, ,$\pi$},
  ytick={0,1},
  max space between ticks=50,
                minor x tick num=3,
                minor y tick num=10,                
  xlabel={\small $x$},
  ylabel={\small $u(x)$},
every axis x label/.style={
below,
at={(2.2cm,0.1cm)},
  yshift=-8pt
  },
every axis y label/.style={
below,
at={(0cm,1.8cm)},
  xshift=-3pt},
  y label style={rotate=90,anchor=south},
  width=4.4cm,
  height=3.6cm,  
  xmin=0,
  xmax=3.14159,
  ymin=0,
  ymax=1] 
\addplot graphics[xmin=0,xmax=3.14159,ymin=0,ymax=1] {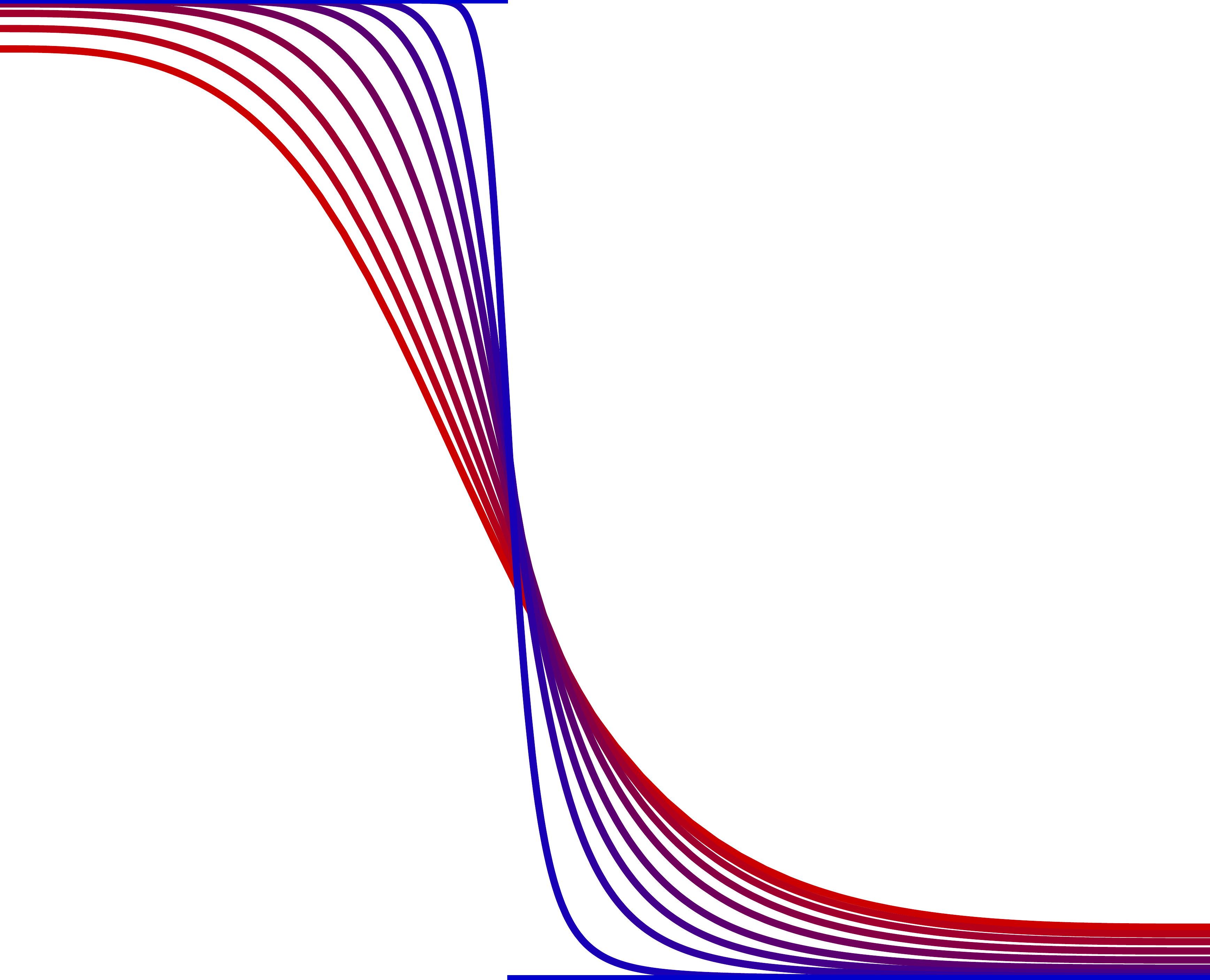};
\end{axis}
\end{tikzpicture}
\caption{Graph of the ``large'' solution $u_{\ell}(x)$ of the Neumann boundary value problem associated with $(\mathscr{E}_{\lambda,\mu})$, where $c=1$, $a(x)$ and $g(u)$ are as in Figure~\ref{fig-1} with $P=\pi$ (and so $m=1$). Notice that $\int_0^{\pi} a(x)\,\mathrm{d}x < 0$. We take $\lambda=\mu\in\{12,$ $15,20,30,50,100,200,500,5000\}$ and we represent also the limit profile.}
\end{subfigure} 
\hspace*{\fill}
\begin{subfigure}[t]{0.45\textwidth}
\centering
\begin{tikzpicture}[scale=1]
\begin{axis}[
  tick label style={font=\scriptsize},
          scale only axis,
  enlargelimits=false,
  xtick={0, 1.31812,3.14159},
  xticklabels={0, ,$\pi$},
  ytick={0,1},
  max space between ticks=50,
                minor x tick num=0,
                minor y tick num=10,                
  xlabel={\small $x$},
  ylabel={\small $u(x)$},
every axis x label/.style={
below,
at={(2.2cm,0.1cm)},
  yshift=-8pt
  },
every axis y label/.style={
below,
at={(0cm,1.8cm)},
  xshift=-3pt},
  y label style={rotate=90,anchor=south},
  width=4.4cm,
  height=3.6cm,  
  xmin=0,
  xmax=3.14159,
  ymin=0,
  ymax=1] 
\addplot graphics[xmin=0,xmax=3.14159,ymin=0,ymax=1] {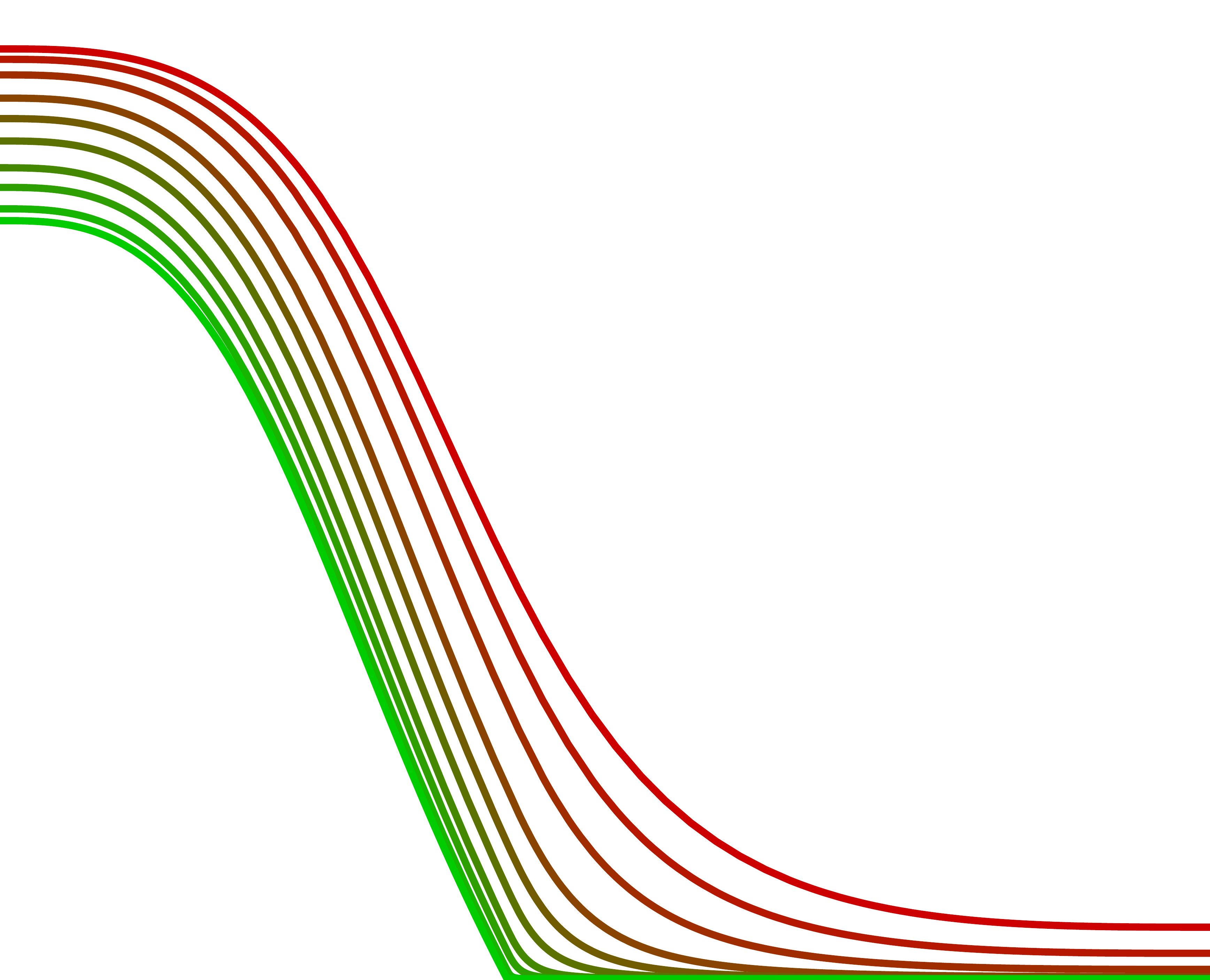};
\end{axis}
\end{tikzpicture}
\caption{Graph of the ``large'' solution $u_{\ell}(x)$ of the Neumann boundary value problem associated with $(\mathscr{E}_{\lambda,\mu})$, where $c=1$, $a(x)$ and $g(u)$ are as in Figure~\ref{fig-1} with $P=\pi$ (and so $m=1$).
We take $\lambda=12$ and $\mu\in\{12,30,100,500,$ $2000,10^{4},10^{5},10^{6},10^{8}\}$ and we represent also the limit profile.}
\end{subfigure} 
\caption{Asymptotic analysis for the indefinite Neumann boundary value problem associated with $(\mathscr{E}_{\lambda,\mu})$ with respect to the parameters $\lambda$ and $\mu$.}     
\label{fig-2}
\end{figure}

\bibliographystyle{elsart-num-sort}
\bibliography{BoFeSo-biblio}

\end{document}